\documentclass[10pt]{article}

\usepackage[letterpaper, left=1in, top=1in, right=1in, bottom=1in, verbose, ignoremp]{geometry}

\usepackage{url}\RequirePackage[colorlinks,citecolor=blue, linkcolor=blue,urlcolor = blue]{hyperref}
\usepackage{latexsym,amssymb,amsmath,amsfonts,graphicx,color,fancyvrb,amsthm,enumerate,subfigure,mathrsfs}
\DeclareMathAlphabet{\mathscrbf}{OMS}{mdugm}{b}{n}
\usepackage[authoryear,round]{natbib}
\usepackage[dvipsnames]{xcolor}
\usepackage{xy}\xyoption{all} \xyoption{poly} \xyoption{knot}
\usepackage{float}
\usepackage{bm}
\usepackage{bbm}
\usepackage{multicol,multirow}
\usepackage{array}
\usepackage{relsize}
\usepackage{chngcntr}
\usepackage{etoolbox}

\usepackage{tikz}
\usetikzlibrary{patterns}

\usepackage{hyperref}

\newtheorem{theorem}{Theorem}[]
\newtheorem*{theorem*}{Theorem}

\newtheorem{proposition}[theorem]{Proposition}

\newtheorem*{claim*}{Claim}
\theoremstyle{definition}
\newtheorem{definition}[theorem]{Definition}
\theoremstyle{AppDefinition}

\theoremstyle{AppClaim}
\newtheorem{AppClaim}{Claim}[]
\theoremstyle{remark}

\newtheorem{example}[theorem]{Example}
\newtheorem*{example*}{Example}

\def\eq#1{(\ref{#1})}

\def\beginmat{ \left( \begin{array} }
\def\endmat{ \end{array} \right) }

\def\log{{\rm log}}

\newcommand*\diff{\mathop{}\!\mathrm{d}}

\newcommand{\T}{\intercal}

\allowdisplaybreaks

\makeatletter
\def\@biblabel#1{}
\makeatother

\makeatletter
\patchcmd{\NAT@citex}
  {\@citea\NAT@hyper@{%
     \NAT@nmfmt{\NAT@nm}%
     \hyper@natlinkbreak{\NAT@aysep\NAT@spacechar}{\@citeb\@extra@b@citeb}%
     \NAT@date}}
  {\@citea\NAT@nmfmt{\NAT@nm}%
   \NAT@aysep\NAT@spacechar\NAT@hyper@{\NAT@date}}{}{}

\patchcmd{\NAT@citex}
  {\@citea\NAT@hyper@{%
     \NAT@nmfmt{\NAT@nm}%
     \hyper@natlinkbreak{\NAT@spacechar\NAT@@open\if*#1*\else#1\NAT@spacechar\fi}%
       {\@citeb\@extra@b@citeb}%
     \NAT@date}}
  {\@citea\NAT@nmfmt{\NAT@nm}%
   \NAT@spacechar\NAT@@open\if*#1*\else#1\NAT@spacechar\fi\NAT@hyper@{\NAT@date}}
  {}{}

\makeatother

\begin{document}
\def\spacingset#1{\renewcommand{\baselinestretch}%
{#1}\small\normalsize} \spacingset{1}
\begin{flushleft}
{\Large{\textbf{Tropical Sufficient Statistics for Persistent Homology}}}
\newline
\\
Anthea Monod$^{1,\dagger}$, Sara Kali\v{s}nik$^{2}$, Juan \'Angel Pati\~{n}o-Galindo$^{1}$, and Lorin Crawford$^{3-5}$
\\
\bigskip
\bf{1} Department of Systems Biology, Columbia University, New York, NY, USA 
\\
\bf{2} Department of Mathematics, Wesleyan University, Middletown, CT, USA
\\
\bf{3} Department of Biostatistics, Brown University, Providence, RI, USA
\\
\bf{4} Center for Statistical Sciences, Brown University, Providence, RI, USA
\\
\bf{5} Center for Computational Molecular Biology, Brown University, Providence, RI, USA
\\
\bigskip
$\dagger$ Corresponding e-mail: am4691@cumc.columbia.edu
\end{flushleft}


\section*{Abstract}
We show that an embedding in Euclidean space based on tropical geometry generates stable sufficient statistics for barcodes.  In topological data analysis, barcodes are multiscale summaries of algebraic topological characteristics that capture the ``shape'' of data; however, in practice, they have complex structures that make them difficult to use in statistical settings.  The sufficiency result presented in this work allows for classical probability distributions to be assumed on the tropical geometric representation of barcodes.  This makes a variety of parametric statistical inference methods amenable to barcodes, all while maintaining their initial interpretations.  More specifically, we show that exponential family distributions may be assumed and that likelihood functions for persistent homology may be constructed.  We conceptually demonstrate sufficiency and illustrate its utility in persistent homology dimensions 0 and 1 with concrete parametric applications to human immunodeficiency virus and avian influenza data.


\section{Introduction}\label{sec:intro}

In this paper, we provide statistical sufficiency in Euclidean space for persistent homology --- an important concept in the field of topological data analysis (TDA) that summarizes the ``shape" and ``size" of data.  Our result is based on a topological embedding given by functions defined in tropical geometry.  With these sufficient summaries, probability distributions from classical statistics may be applied to persistent homology.  More importantly, extensive statistical methodology and parametric inference have become accessible to TDA.

Recently, TDA has become particularly relevant due to its theoretical foundations that allow for dimensionality reduction with qualitative and robust summaries of observed data.  To this end, it is applicable to a wide range of complex data structures that arise in various domains of data science.  Persistent homology is an important topological invariant upon which many TDA methods depend when applied in practice \citep{ZC}.  Specifically, this method has been used to address problems in fields ranging from sensor networks~\citep{VinEvader, adams}, medicine~\citep{Ferri, ADCOCK201436}, neuroscience~\citep{Chung2009, Curto2013, Giusti_Pastalkova_Curto_Itskov_2015}, as well as imaging analysis~\citep{Klein}.

The output of persistent homology is a collection of intervals (known as barcodes, or persistence diagrams, in their representation as points on a plane) rather than numerical quantities.  This has made it difficult to apply the method to parametric data analysis.  Recently, there have been substantial advancements in applied topology to bypass this issue.  More recent work entails developing statistical methodology and machine learning algorithms to work directly with barcodes and persistence diagrams.  Examples of these include confidence sets separating topological noise from topological signal on a persistence diagram \citep{fasy2014}; a positive-definite multiscale kernel for persistence diagrams \citep{Reininghaus_2015_CVPR}; and an input layer to deep neural networks to take in persistence diagrams directly \citep{NIPS2017_6761}.  An alternative approach entails manipulating barcodes so that they may be integrated into existing computational machinery, predominantly by vectorization, which is the focus of this paper.  Various such techniques have been developed with specific properties depending on the goal of the study.  A notable example is the persistence landscape, which assigns functions directly to barcodes and produces elements in Banach space \citep{Bubenik:2015:STD:2789272.2789275}.  The persistence landscape was then generalized to the persistence silhouette by \cite{Chazal:2014:SCP:2582112.2582128}, and to the kernel-smoothed persistence intensity function by \cite{chen2015statistical}, for statistical analyses such as clustering and weak convergence of bootstrapped persistence summaries.  Additionally, feature representations are constructed by \cite{bartovecInria}, by rearranging entries of a distance matrix between points in a persistence diagram.  Another approach proposed by \cite{7855573} uses binning to obtain a vector representation of features (though there are concerns of stability).  Similar methods where stable surfaces of persistence diagrams are constructed, based on kernel methods, with the aim of incorporating persistence information into support vector machines and other machine learning algorithms are studied by \cite{Reininghaus_2015_CVPR}, \cite{NIPS2015_5887}, and \cite{Adams:2017:PIS:3122009.3122017}.  Persistence information may also be integrated as vectors into functional regression models via topological summary statistics based on Euler characteristics \citep{1611.06818}.

Summary statistics are relevant to the notion of sufficiency, which is a desirable property in studies centered on inference.  Sufficiency allows a given sample of data to be mapped to a less complex or lower dimensional space (i.e., harboring less computational burden) without the loss of information \citep{Bartlett:1937aa}.  Moreover, sufficient statistics provide the functional form of probability distributions via a classical factorization criterion.  In other words, they represent the key basis for complete parametric inference in statistics.  Summary statistics have been explored in the context of persistent homology to study shapes and surfaces (closed, compact subsets) by \cite{Turner01122014}: specifically, persistence diagrams were used to summarize and compare heel bones of various species of primates.  Their construction provided a notion of summary statistics for the family of probability distributions on shapes and surfaces in 2 and 3 dimensions, in the space of persistence diagrams.  Perhaps more importantly, this effort also shed light on the possibility of proving statistical sufficiency in the context of topological analyses of data.


The main contribution of our work is to provide statistical sufficiency for the general family of probability measures on persistent homology.  A fundamental observation that makes this result possible is the formal probabilistic characterization of both the domain (i.e., the space of barcodes) and codomain (i.e., Euclidean space) of our mapping \citep{0266-5611-27-12-124007}.  In principle, inference should proceed from a probability model defined directly on the barcodes.  However, this type of specification is complicated due to the complex geometry of the set o fall barcodes.  Given sufficient statistics for barcodes, we may use the likelihood principle to proceed with parametric inferences directly.  As has been recently demonstrated, this suggests that a generative or sampling model on barcodes is possible in persistent homology \citep{1704.08248}.

The remainder of this paper is organized as follows.  In Section \ref{sec:barcodes}, we give a formal definition of persistent homology and outline the approach to map persistence barcodes to vectors using functions defined via tropical geometry.  We also discuss properties of these functions.  In Section \ref{sec:suff}, we provide our main result of sufficient statistics for persistence barcodes.  Based on this result, we given the form of an exponential family of probability distributions and likelihood functions for persistent homology.  Section \ref{sec:app} gives two concrete applications to data in dimensions 0 and 1.  The first is a practical demonstration of sufficiency based on human immunodeficiency virus (HIV) data, while the second shows how a parametric assumption statistically quantifies the biological distinction between intra- and intersubtype reassortment in avian influenza.  We close with a discussion in Section \ref{sec:discussion} with directions for future research.




\section{Persistent homology and vectorization}
\label{sec:barcodes}

In this section, we give the mathematical background to our main result.  We provide details on the space of barcodes arising from persistent homology and describe the construction of tropical coordinates on the space of persistence barcodes.

\subsection{The space of persistence barcodes}
\label{subsec:tda}

Homology groups were developed in classical topology to ``measure" the shape of spaces by abstractly counting the occurrences of patterns, such as the number of connected components, loops, and voids in 3 dimensions.  The primary data analytic utility of {\em persistent homology} or {\em persistence}~\citep{Frosini_sizetheory, elz-tps-02, ZC} --- an important concept from TDA that is heavily leveraged in many data applications --- is the adaptation of classical homology to point clouds (i.e., finite metric spaces), which is often the form in which data are collected.  Persistence, as a topological construct, produces a robust summary of the data that is invariant to noisy perturbations \citep{Cohen-Steiner2007}.  Moreover, it also captures integral geometric features (``size") of the data.  We now describe the construction of homology and its extension to persistent homology.


\paragraph{Simplicial homology and persistence.}


Simplicial homology studies the shape of simplicial complexes, which can be seen as skeletal representations of data types.  An example of a simplicial complex is a mesh, which can be viewed as a discretization of an image.  Simplicial complexes are an important building block in computing persistent homology: an abstract {\em simplicial complex} is a collection $K$ of nonempty subsets of a set $K_0$ such that $\tau \subset \sigma$ and $\sigma \in K$ guarantees that $\tau \in K$.  The elements of $K_0$ are called {\em vertices} of $K$, and the elements of $K$ are called {\em simplices}.  A simplex $\sigma$ has dimension $k$, or is a $k$-{\em simplex}, if it has cardinality $k + 1$. 

Simplicial homology is constructed by considering simplicial {\em $k$-chains}, which are linear combinations of $k$-simplices in finite $K$ over a field $\mathbb{F}$.  A set of $k$-chains defines a vector space $C_k(K)$.  When $\mathbb{F}$ is the binary field $\mathbb{Z}/2\mathbb{Z}$, the {\em boundary map} is
\begin{align*}
\partial_k: C_k(K) & \rightarrow C_{k-1}(K),\\
\partial_k\big( [ v_0, v_1, \ldots, v_k] \big) & = \sum_{i=0}^k [v_0, \ldots, v_{-i}, \ldots, v_k]
\end{align*}
with linear extension, where $v_{-i}$ indicates that the $i$th element has been dropped.  {\em Boundaries} are elements of $B_k(K) = \mathrm{im}\ \partial_{k+1}$, and {\em cycles} are elements of $Z_k(K) = \mathrm{ker}\ \partial_{k}$.

\begin{definition}
The {\em $k$th homology group} of $K$ is given by the quotient group
$$
H_k(K) := Z_k(K)/B_k(K).
$$
\end{definition}

The motivating idea underlying homology is to account for the structure of $K$, which is a finite simplicial complex representation of a topological space $X$.  In dimension 0, elements representing connected components of $X$ with finite simplicial complex representation $K$ generate the zeroth homology group $H_0(X)$.  If, for example, $X$ has two connected components, then $H_0(X) \cong \mathbb{F} \oplus \mathbb{F}$ (where $\cong$ denotes group isomorphism, and $\oplus$ here denotes the direct sum).  For higher dimensions $k \geq 1$, $k$-dimensional holes are the result of considering the boundary of a $(k+1)$-dimensional object.  $H_k(X)$ is generated by elements that represent $k$-dimensional holes in $X$.  The rank of the $k$th homology group is denoted by $\beta_k$ and referred to as the $k$th Betti number.



Intuitively, persistent homology is computed by tracking the progression of connected components, loops, and higher dimensional voids, with respect to a {\em filtration} assigned to the observed point cloud.  A filtration is a finite nested sequence of simplicial complexes $\mathcal K := \{ K_r \}_{r=a}^{b}$ indexed by a parameter $r \in \mathbf{R}^+$ such that $K_{r_1} \subseteq K_{r_2}$ if $r_1 < r_2$.


\begin{definition}
Let $K$ be a filtered simplicial complex $K_0 \subset \cdots \subset K_t = K$.  The {\em $k$th persistence module derived in homology},  or simply {\em $k$th persistent homology}, of $K$ is given by
$$
\mathrm{PH}_k(K) := \big\{ H_k(K_r) \big\}_{0 \leq r \leq t}
$$ 
together with the collection of linear maps $\{ \varphi_{r, s} \}_{0 \leq r \leq s \leq t}$, where $\varphi_{r,s}: H_k(K_r) \rightarrow H_k(K_s)$ is induced by the inclusion $K_r \hookrightarrow K_s$ for all $r, s \in [0,t]$ such that $r \leq s$.
\end{definition}

Persistent homology contains the homology information on individual spaces $\{ K_r \}$ and on the mappings between their homologies for every $K_r$ and $K_s$, where $r \leq s$.  Note that the motivating idea underlying persistent homology is to account for the homology of $X$ simultaneously across multiple scales.  Rather than restricting the analysis to one static instance, persistence tracks how the topological structure evolves over the indexed filtration.  The final output after computing persistence is a \emph{barcode} (i.e.~a collection of intervals).  Each interval (or bar) corresponds to a topological feature that appears (i.e.~birth) at the value of a parameter given by the left endpoint of the interval, and disappears or merges with another existing feature (i.e.~death) at the value given by the right endpoint.  Figure \ref{Fig_A4} provides an illustrative example of persistent homology.

\paragraph{Barcodes and persistence diagrams.}

For a given dimension $k$, a barcode consisting of $n$ intervals, obtained from computing the $k$th persistent homology, can be written as a collection $(x_1, d_1, x_2, d_2,\ldots, x_n, d_n)$, where $x_i$ is the left endpoint of the $i$\textsuperscript{th} interval and $d_i$ is its length.  In this paper, we assume that $x_i \geq 0$, thus meaning that the birth times of bars are nonnegative.  Notice that the order in which coordinates $(x_i, d_i)$ appear within the collection $(x_1, d_1, x_2, d_2, \ldots, x_n, d_n)$ does not affect the content of the topological information encoded in the barcode.  Hence, the indices $i=1, 2, \ldots, n$ are essentially dummy variables.  This means that all permutations of coordinates $(x_i, d_i)$ across $n$ positions define the same collection.  Algebraically, this amounts to taking the orbit space of the action of the symmetric group on $n$ letters on the product $\big([0, \infty) \times [0, \infty) \big)^n$, given by permuting the coordinates.  We denote this orbit space by $B_n$.

\begin{definition}
The {\em barcode space} $\mathcal{B}_{\leq n}$ consisting of bars with at most $n$ intervals is given by the quotient
\[
\coprod_{n\in \mathbf{N}_{\leq n}} B_n /_{\sim},
\]
where $\sim$ is generated by equivalences of the form 
\[
\big\{(x_1, d_1), (x_2, d_2), \ldots, (x_n, d_n) \big\} \sim \big\{(x_1, d_1), (x_2, d_2), \ldots, (x_{n-1}, d_{n-1}) \big\},
\]
whenever $d_n=0$. 
\end{definition}

We may set $y_i := x_i + d_i$ to be the right endpoint of the $i$th interval, and equivalently consider the collection of ordered pairs $(x_i, y_i)$ together with the diagonal $\Delta = \big\{ (x, y) \mid x = y \big\}$, where each point is taken with infinite multiplicity.  Here, $\Delta$ contains the bars of length zero (i.e., features that are born and die at exactly the same time and therefore have zero persistence).  Plotting these ordered pairs on a scatterplot, we obtain an alternate representation of the barcode, commonly known as a {\em persistence diagram}.

\paragraph{Metrics on barcode space.}

The space of barcodes can be equipped with various measures of distance, which, under mild regularity conditions (namely, local finiteness of barcodes or persistence diagrams), are also metrics.  In this paper, we consider the bottleneck and Wasserstein distances, which are the distance functions used in many existing results in the field of TDA.  In particular, the stability theorem \citep{Cohen-Steiner2007}, which provides a foundation that makes persistent homology applicable to real data analysis, is stated with respect to the bottleneck distance.  Unless explicitly stated otherwise, all results in this paper hold in generality for all metrics discussed here --- namely, for the Wasserstein $p$-distance for all $p \geq 1$, as well as the bottleneck distance.

The first step in specifying the distance between any two barcodes is to define the distance between any pair of intervals, as well as the distance between any interval and the set of zero length intervals $\Delta$.  We take
\[
\textrm{d}_\infty \big((x_i, d_i), (x_j, d_j) \big) = \max \big\{ |x_i-x_j|,\, |d_i-d_j +x_i- x_j| \big\}
\]
to be the distance between barcodes $\{(x_i, d_i)\}$ and $\{(x_j, d_j)\}$.  The distance between an interval $(x, d)$ and the set $\Delta$ is
\[
\textrm{d}_\infty \big((x, d), \Delta \big) = \frac{d}{2},
\]
which is simply the minimal $\textrm{d}_\infty$ distance between $(x, x+d)$ and $\Delta$.  Now let $\mathscr{B}_1 = \{I_\alpha\}_{\alpha \in A}$ and $\mathscr{B}_2 = \{J_\beta\}_{\beta \in B}$ be two barcodes. 

\begin{definition}
For a fixed $p$, finite sets $A$ and $B$, any bijection $\vartheta$ from a subset $A' \subseteq A$ to $B' \subseteq B$, and a penalty on $\vartheta$ set to
\[
P_p(\vartheta) := \sum_{a\in A'}\textrm{d}_\infty\big(I_a, J_{\vartheta(a)} \big)^p +\sum_{a\in A \setminus A'} \textrm{d}_\infty \big(I_a, \Delta \big)^p +\sum_{b\in B\setminus B'} \textrm{d}_\infty \big(I_b, \Delta \big)^p,
\] 
the {\em Wasserstein $p$-distance} ($p\geq 1$) between $\mathscr{B}_1$ and $\mathscr{B}_2$ is given by
\[
d^{W}_p(\mathscr{B}_1, \mathscr{B}_2) := \Big(\min_\vartheta P_p(\vartheta) \Big)^{\frac{1}{p}}.
\]
\end{definition}

\begin{definition}
For finite subsets and sets $A' \subseteq A$ and $B' \subseteq B$, a bijection $\vartheta$ as previously defined above, and a penalty on $\vartheta$ set to
$$
P_\infty(\vartheta) := \max\left\{ \max_{a \in A'} \Big( \textrm{d}_\infty \big(I_a, J_{\vartheta(a)}\big) \Big),~\max_{a\in A\setminus A'} \textrm{d}_\infty(I_a, \Delta),~\max_{b \in B\setminus B'} \textrm{d}_\infty (I_b, \Delta) \right\},
$$
the quantity
$$
d^B_\infty(\mathscr B_1, \mathscr B_2) := \min_\vartheta P_\infty(\vartheta),
$$
defines the {\em bottleneck distance} between $\mathscr B_1$ and $\mathscr B_2$, where the minimum is taken over all bijections from subsets $A'$ to $B'$.
\end{definition}

\paragraph{Barcodes  and data.}

In applications, data are generated by random processes and thus their representation as simplicial complexes and their resulting barcodes are also considered to be random objects.  Barcodes can therefore be seen as summary statistics of the data generating process because they allow for a dimensionality reduction of the ambient space.  Therefore, in a broader and more philosophical sense, topological approaches in data science applications can be used to provide sets of principal variables for scalable inference.  In practice, however, estimating persistent homology is subject to the curse of dimensionality, since it amounts to statistically estimating pairwise correspondences of sets in Hausdorff distance.

\subsection{Coordinates on the space of persistence barcodes}
\label{subsec:coords}

To integrate shape information of data (via barcodes) into existing computational machinery and methodology, several vectorization techniques have been proposed.  The vectorization that \cite{TropCoord} proposes takes the form of functions, which we build upon and extend to statistical theory in this paper. 

For the goal of relating persistent homology to statistical theory, studying functional vectorization techniques is preferable to studying algorithmic vectorization, because specific properties relevant to statistical inference, such as injectivity and measurability, are assessed on functions rather than algorithms.  In addition to the vectorization of \cite{TropCoord}, there are other vectorization approaches based on the identification of polynomials or functions.  For example, complex polynomials (where the points of a persistence diagram are the roots) have been developed by \cite{landi}; however, these are difficult to apply to parametric statistical analysis where data are primarily real-valued.  Persistence landscapes by \cite{Bubenik:2015:STD:2789272.2789275} are also less natural for statistical inference due to certain characteristics of Banach space (for example, the fact that norms on Banahc space need not comprise an inner product structure).

Another particularly relevant approach by \cite{algfn} assigns vectors to barcodes directly by defining a ring of algebraic functions.  These functions are polynomials and satisfy desirable properties such as symmetry in the variables (i.e., invariance of ordering of the bars), and invariance under addition of 0-length intervals.  Unfortunately, it turns out that these functions are not Lipschitz with respect to the Wasserstein $p$- and bottleneck distances (see Appendix \ref{appendix:math} for a proof).  This lack of Lipschitz continuity is problematic in real data applications that are generated by random processes because of its negative implications on stability in the target space under perturbations in the domain.  More specifically, there are no guarantees that the transformation is computationally or statistically robust.  As a viable solution to address such shortcomings of formerly defined polynomials by \cite{algfn}, tropical functions on the space of barcodes were identified.  These tropical functions possess the same desirable properties as those developed by \cite{algfn}, but additionally, were shown to be Lipschitz (and therefore continuous) with respect to the Wasserstein $p$-distance for $p\geq 1$ and the bottleneck distance \citep{TropCoord}.

We now briefly review the definition of the tropical semiring and the tropical functions that we study as coordinates on barcode space.


\paragraph{Fundamentals of tropical algebra.}

Tropical algebra is a branch of mathematics based on the study of the tropical semiring.

\begin{definition} 
The {\em tropical (equivalently, min-plus) semiring} is $(\mathbb{R} \cup \{+\infty \}, \oplus, \odot)$, with addition and multiplication being defined as
\[
\begin{array}{ccc}
a\oplus b := \min{\{ a, b \}} &\, \textrm{and} \,& a\odot b := a+b,
\end{array}
\]
where both operations are commutative and associative, and multiplication distributes over addition.  Similarly, there exists the {\em max-plus semiring} $(\mathbb{R} \cup \{-\infty \}, \boxplus, \odot)$, where multiplication of two elements is defined as in the case of the tropical semiring, but addition amounts to taking the maximum instead of the minimum.  Namely,
\[
\begin{array}{ccc}
a\boxplus b := \max{\{ a, b \}} &\, \textrm{and} \,& a\odot b := a+b,
\end{array}
\]
where again both operations are associative, commutative, and distributive as in the tropical semiring.
\end{definition}

As in the case of ordinary polynomials which are formed by multiplying and adding real variables, max-plus polynomials can be formed in a similar fashion as follows.  Let $x_1, x_2, \ldots, x_n$ be variables that are elements in the max-plus semiring.  A \emph{max-plus monomial expression} is any product of these variables, where repetition is permitted.  By commutativity, we can sort the product and write monomial expressions with the variables raised to exponents:
\begin{align*}
p(x_1, x_2, \ldots, x_n) & = a_1\odot x_1^{a_1^1} x_2^{a_2^1} \cdots x_n^{a_n^1} \boxplus a_2\odot x_1^{a_1^2} x_2^{a_2^2} \cdots x_n^{a_n^2}\boxplus \cdots \boxplus a_m\odot x_1^{a_1^m} x_2^{a_2^m} \cdots x_n^{a_n^m}\\
& = \max(a_1 + a_1^1 x_1 + \cdots + a_n^1 x_n,\, \cdots,\, a_m + a_1^m x_1 + \cdots + a_n^m x_n).
\end{align*}
Here, the coefficients $a_1, a_2, \ldots, a_m$ are in $\mathbb{R}$, and the exponents $a_j^i$ are in $\mathbb{Z}^+$ for ${1\leq j \leq n}$ and ${1\leq i \leq m}$.  Max-plus polynomial expressions do not uniquely define functions \citep[e.g.,][]{KalisnikSym}, which suggests an equivalence relation.  For max-plus polynomial expressions $p$ and $q$, if
\[
p(x_1, x_2, \ldots, x_n) = q(x_1, x_2, \ldots, x_n)
\]
for all $(x_1, x_2, \ldots, x_n)\in (\mathbb{R} \cup \infty)^n$, then $p$ and $q$ are said to be \emph{functionally equivalent}, and we write $p \sim q$.  Max-plus polynomials are the semiring of equivalence classes for max-plus polynomial expressions with respect to $\sim$. 

\paragraph{Identifying tropical functions for barcodes.}

We now give a list of tropical functions that we use to assign vectors to barcodes, restate existing relevant results, and prove additional statements necessary for our main result (see Theorem~\ref{sufficient} in Section \ref{sec:suff}).  For a complete discussion on the construction of the functions, and proofs related to the original definitions of the tropical functions, see the previous works: \cite{KalisnikSym, TropCoord}. 

Fix $n$ and let the symmetric group $S_n$ act on the following matrix of indeterminates by left multiplication:
\[
X = \begin{pmatrix} 
x_{1, 1} & x_{1, 2} \\
x_{2, 1} & x_{2, 2} \\
\vdots & \vdots \\
x_{n, 1} & x_{n, 2} \\
\end{pmatrix}.
\]
Here, the matrix $X$ represents a barcode with $n$ bars, where each bar represented by a row.  As previously mentioned, the action of $S_n$ on $X$ amounts to permuting the order of the bars within $(x_1, d_1, x_2, d_2, \ldots, x_n, d_n)$.  In parallel, also consider the collection of {\em exponent} matrices
\[
\mathscr{E}_n = \left\{ \begin{pmatrix} 
e_{1, 1} & e_{1, 2} \\
e_{2, 1} & e_{2, 2} \\
\vdots & \vdots \\
e_{n, 1} & e_{n, 2} \\
\end{pmatrix} \neq [0]_n^2 \,\mid\, e_{i, j} \in \{0, 1\} \textrm{ for }i=1,2,\ldots, n, 
\textrm{ and } j=1,2  \right\}.
\]
A matrix $E \in \mathscr{E}_n$, together with $X$, determines a max-plus monomial by $P(E) = x_{1, 1}^{e_{1,1}} x_{1, 2}^{e_{1,2}} \cdots x_{n, 1}^{e_{n,1}} x_{n, 2}^{e_{n, 2}}$.  Now, denote the set of orbits under the row permutation action on $\mathscr{E}_n$ by $\mathscr{E}_n/{S_n}$.  Each orbit $\{E_1, E_2, \ldots, E_m\}$ determines an {\em elementary 2-symmetric} max-plus polynomial by \[P(E_1)\boxplus P(E_2)\boxplus \cdots \boxplus P(E_m). \]
Let $E_{(e_{1, 1}, e_{1, 2}), (e_{2,1}, e_{2,2}), \ldots, (e_{n, 1}, e_{n, 2}) }$ denote the polynomial that arises from the orbit
$$
\left[\begin{pmatrix} 
e_{1, 1} & e_{1, 2} \\
e_{2, 1} & e_{2, 2} \\
\vdots & \vdots \\
e_{n, 1} & e_{n, 2} \\
\end{pmatrix}\right].
$$

\begin{example}\label{notation}
Let $n=3$.  Here are a few examples of elementary 2-symmetric max-plus polynomials:
\[
\begin{array}{ccccc}
E_{(0,1), (0,1), (0, 0)} (x_1, d_1, x_2, d_2, x_3, d_3) &=& d_1\odot d_2 \boxplus d_1\odot d_3 \boxplus d_2\odot d_3 &=& \max\{d_1+d_2,\, d_1+d_3,\, d_2+d_3 \}.\\
E_{(0,1), (0, 0), (0,1)} (x_1, d_1, x_2, d_2, x_3, d_3) &=& d_1\odot d_2 \boxplus d_1\odot d_3 \boxplus d_2\odot d_3 &=& \max\{d_1+d_2,\, d_1+d_3,\, d_2+d_3 \},\\
E_{(1, 1), (1, 1), (1, 1)} (x_1, d_1, x_2, d_2, x_3, d_3) &=& x_1 \odot d_1  \odot x_2  \odot d_2 \odot x_3\odot d_3 &=& x_1+d_1+x_2 +d_2+x_3+d_3.\\
\end{array}
\]
\end{example}

In the example above, the first two polynomials are equivalent: $E_{(0, 1), (0, 1), (0, 0)}$ is the same function as $E_{(0, 1), (0, 0), (0, 1)}$, which in turn is the same as $E_{(0, 0), (0, 1), (0, 1)}$.  Thus, only $n$ and the number of $(0, 1)$ rows are needed to determine the 2-symmetric polynomial in this example.  Since we know $n$, we may therefore simplify notation by writing this polynomial as $E_{(0, 1)^2}$ instead of listing every row.  This is the notation convention that we will adopt throughout the rest of the paper.  Namely, whenever $n$ is known, we write $E_{(e_{1, 1}, e_{1, 2}), \ldots, (e_{n, 1}, e_{n, 2})}$ as $E_{(1, 0)^k, (0, 1)^i, (1, 1)^j}$---where among $(e_{1, 1}, e_{1, 2}), (e_{2,1}, e_{2,2}), \ldots, (e_{n, 1}, e_{n, 2})$, $k$ is the number of $(1, 0)$ rows, $i$ the number of $(0, 1)$ rows, and $j$ the number of $(1, 1)$ rows.  The number of $(0, 0)$ rows is what remains (i.e., $n-i-j-k$), which we omit from the notation.  If we wish to emphasize the value of $n$, we write $E^n_{(1, 0)^k, (0, 1)^i, (1, 1)^j}$.

These functions are the building blocks for the functions we use to vectorize the barcode space.  Proposition 5.8 by~\cite{KalisnikSym} asserts that $E_{(0, 1)^i, (1, 1)^j, (1,0)^k}$ separate orbits under the row permutation action on $\mathbb{R}^{2n}$.  In other words, if $\mathscr{B}_1 \neq \mathscr{B}_2$, then $\{i, j, k\}$ exist such that $E_{(0, 1)^i, (1, 1)^j, (1,0)^k}(\mathscr{B}_1) \neq E_{(0, 1)^i, (1, 1)^j, (1,0)^k}(\mathscr{B}_2)$.  We improve this statement and show that the two-parameter family of $E_{(0, 1)^i, (1, 1)^j}$, parametrized by $i$ and $j$ alone, suffices to separate the barcodes, and that the $k$-factor of $(1, 0)$ rows is not needed.  This is an important result, as it enables us to embed the barcode space into a lower-dimensional Euclidean space, and hence tightens the main sufficiency result of this paper.  We now formalize this fact in the following proposition.

\begin{proposition}\label{minset}
Let $\big[(x_1, d_1, x_2, d_2,\ldots, x_n, d_n) \big]$ and $\big[(x_1', d_1', x_2', d_2', \ldots , x_n', d_n') \big]$ be two orbits under the row permutation action on $\mathbb{R}^{2n}$. 
If
\[
E^n_{(0, 1)^i, (1, 1)^j} \big[(x_1, d_1, x_2, d_2, \ldots, x_n, d_n) \big] = E^n_{(0, 1)^i, (1, 1)^j} \big[(x_1', d_1', x_2', d_2', \ldots , x_n', d_n') \big]
\]
for all $i, j \leq n$, then
\[
\big[(x_1, d_1, x_2, d_2, \ldots, x_n, d_n) \big] = \big[(x_1', d_1', x_2', d_2', \ldots , x_n', d_n') \big].
\]
In other words, $E^n_{(0, 1)^i, (1, 1)^j}$ separates barcodes.
\end{proposition}

\begin{proof}
We prove the statement by induction on the number of nonzero bars $\leq n$.

If $n=1$, then $E^1_{(0, 1)}\big[(x_1, d_1) \big] = E^1_{(0, 1)}\big[(x_1', d_1') \big]$ implies that $d_1 = d_1'$, and $\big(E^1_{(0, 1)}-E^1_{(1, 1)} \big)\big[(x_1, d_1) \big] = \big(E^1_{(0, 1)}-E^1_{(1, 1)} \big)\big[(x_1', d_1') \big]$ implies that $x_1 = x_1'$.  From here it follows that $[(x_1, d_1)] = [(x_1', d_1')]$, and the base case holds.

Now suppose that $E^{n-1}_{(0, 1)^i, (1, 1)^j}$ for nonnegative integers $i, j$, where $i+j\leq n-1$, separates barcodes with up to $(n-1)$ nonzero bars.  We must show that $E^{n}_{(0, 1)^i (1, 1)^j}$ separates barcodes with up to $n$ bars.  Let $\mathscr{B}_1=[(x_1, d_1, \ldots, x_n, d_n)]$ and $\mathscr{B}_2=[(x_1', d_1', \ldots, x_n', d_n')]$ be two barcodes with up to $n$ nonzero bars.  Assume without loss of generality that $x_1 \leq x_2 \leq \cdots \leq x_n$ and $x_1' \leq x_2' \leq \cdots \leq x_n'$.  We furthermore arrange that for all $i\leq n-1$ with $x_i=x_{i+1}$, we have $d_i\leq d_{i+1}$.

If $E^n_{(0, 1)} (\mathscr{B}_1) = E^n_{(0, 1)} (\mathscr{B}_2),$ it follows that $\max_{1\leq i\leq n} d_i = \max_{1\leq i\leq n} d_i'$.  Since 
$$
E^n_{(0, 1)^2} (\mathscr{B}_1) - E^n_{(0, 1)^1}  (\mathscr{B}_1) = E^n_{(0, 1)^2}  (\mathscr{B}_2)- E^n_{(0, 1)^1}  (\mathscr{B}_2),
$$
it follows that the second greatest element in $\{ d_i \mid 1\leq i\leq n \}$ is equal to the second greatest element in ${\{ d_i' \mid 1\leq i\leq n \}}$.  We apply this reasoning to the remaining elements in decreasing order and finally obtain $\{ d_i \mid 1\leq i\leq n \} = \{ d_i' \mid 1\leq i\leq n \}$.

The fact that $x_n$ equals $x_n'$ follows by applying $E^n_{(0, 1)^{n-1}, (1, 1)^1}$.  Computing the difference
$$
{E^n_{(0, 1)^{n-2}, (1, 1)^2}-E^n_{(0, 1)^{n-1}, (1, 1)^1}}
$$
yields $x_{n-1}=x_{n-1}'$.  We continue in this manner and finally arrive at the conclusion that $\{ x_1, x_2, \ldots, x_n\}$ is a permutation of $\{ x_1',  x_2', \ldots, x_n'\}$.  This means that unless
$
\big[(x_1, d_{\pi(1)}, \ldots, x_n, d_{\pi(n)}) \big] =\big[(x_1, d_1, \ldots, x_n, d_n) \big]
$
for some permutation $\pi$, one of the functions among $E^{n}_{(0, 1)^i, (1, 1)^j}$ separates $\mathscr{B}_1$ and  $\mathscr{B}_2$.

Assume now that we have such a $\pi$ where $\big[(x_1, d_{\pi(1)}, \ldots, x_n, d_{\pi(n)}) \big] =\big[(x_1, d_1, \ldots, x_n, d_n) \big]$.  Let $I = \{k \in \{ 1, 2, \ldots, n \} \mid d_{k} > d_{\pi(n)} \}$. 
Using the fact that $x_1 \leq x_2 \leq \cdots \leq x_n$, we evaluate $E^n_{ (0, 1)^{|I|},(1,1) }$, where $|I|$ is the cardinality of $I$:
\[
E^n_{(0,1)^{|I|}(1,1)}\big(\big[(x_{1}, d_{\pi(1)}, \ldots, x_{n}, d_{\pi(n)}) \big] \big) = x_n+ d_{\pi(n)} +  \sum_{k\in I}d_{k}.
\]
After splitting the index set according to membership in $I$, we arrive at
\[
E^n_{(0,1)^{|I|},(1,1)}\big( \big[(x_{1}, d_{1}, \ldots, x_{n}, d_n) \big] \big) = \max\bigg( \max_{ \ell \notin I } \big(x_\ell + \sum_{k\in I}d_k + d_{\ell} \big),\, \max_{ \ell \in I} \big(x_\ell + \sum_{k\in I}d_k + d_{\pi(n)} \big) \bigg).
\]
So if $d_n < d_{\pi(n)}$, then also $d_{n-i},d_{n-i+1}, \ldots < d_{\pi(n)}$, which implies that $n-i, n-i+1, \ldots , n \notin I$.  Thus, if $\ell \notin I$, then there are two cases.  First, $\ell<n-i$, in which case $x_\ell + d_\ell \leq x_{n-i-1} + d_{\pi(n)} <  x_n + d_{\pi(n)}$, giving rise to the inequality 
\[
E^n_{(0,1)^{|I|},(1,1)}\big( \big[(x_{1}, d_{1}, \ldots, x_{n}, d_n) \big] \big) < E^n_{(0,1)^{|I|},(1,1)}\big( \big[(x_{1}, d_{\pi(1)}, \ldots, x_{n}, d_{\pi(n)}) \big] \big).
\]
Second, if $n-i \leq \ell \leq n$, then $x_\ell + d_\ell = x_n + d_\ell < x_n + d_{\pi(n)}$ again.  As for the second max, if $\ell\in I$, then necessarily $\ell < n-i$, so $x_\ell + d_{\pi(n)} < x_n + d_{\pi(n)}$.  A symmetrical argument shows that one can distinguish between the barcodes with $E^n_{(0, 1)^{|I|},(1,1)}$, where $I = \{k \in \{ 1, 2, \ldots, n \} \mid d_{k} > d_{n} \}$ if ${ d_n > d_{\pi(n)}}$.  So unless the maximal $x$ values are not matched with the same $d$ values in $\mathscr{B}_1$ and $\mathscr{B}_2$, we can distinguish the barcodes using elementary 2-symmetric polynomials.

Now suppose that $x_n'=x_n$ and $d_n' = d_n$.  We show using the inductive hypothesis that if 
\[
{\mathscr{B}_1\setminus \{ (x_n, d_n) \} \neq \mathscr{B}_2\setminus \{ (x_n, d_n) \}},
\]
then we can distinguish $\mathscr{B}_1$ and $\mathscr{B}_2$ by elementary 2-symmetric polynomials in $2n$ variables, which we construct from elementary 2-symmetric polynomials in $2(n-1)$ variables.  By symmetry between $\mathscr{B}_1$ and $\mathscr{B}_2$, we may assume that there exist $i$ and $j$ such that
\[
E^{n-1}_{(0,1)^{i}, (1,1)^{j}} \big(\mathscr{B}_1\setminus \{ (x_n, d_n) \} \big) > E^{n-1}_{(0,1)^{i}, (1,1)^{j}} \big(\mathscr{B}_2\setminus \{ (x_n, d_n) \} \big).
\]
Note that $E^{n-1}_{(0,1)^{i}, (1,1)^{j}} \big(\mathscr{B}_k\setminus \{ (x_n, d_n) \} \big) \leq E^{n}_{(0,1)^{i}, (1,1)^{j}}(\mathscr{B}_k)$ for $k=1, 2$. 

If $E^{n}_{(0,1)^{i}, (1,1)^{j}} \big(\mathscr{B}_1 \big) \neq E^{n}_{(0,1)^{i}, (1,1)^{j}} \big(\mathscr{B}_2 \big)$, we are done, since we have a function that separates the two barcodes.  If $E^{n-1}_{(0,1)^{i}, (1,1)^{j}}\big(\mathscr{B}_2\setminus \{ (x_n, d_n) \} \big) = E^{n}_{(0,1)^{i}, (1,1)^{j}}\big(\mathscr{B}_2 \big)$, then
$
{E^{n}_{(0,1)^{i}, (1,1)^{j}}(\mathscr{B}_1)> E^{n}_{(0,1)^{i}, (1,1)^{j}}(\mathscr{B}_2)}
$
and thus $E^{n}_{(0,1)^{i}, (1,1)^{j}}$ distinguishes $\mathscr{B}_1$ and $\mathscr{B}_2$. It then suffices to study the situation when
\begin{equation}\label{eq1}
E^{n-1}_{(0,1)^{i}, (1,1)^{j}}\big(\mathscr{B}_2\setminus \{ (x_n, d_n) \} \big) < E^{n}_{(0,1)^{i}, (1,1)^{j}}(\mathscr{B}_2).
\end{equation}
In this situation $E^{n}_{(0,1)^{i}, (1,1)^{j}}$ must depend on at least $d_n$ or $x_n$ (or possibly both). 
Then in the partition $I, J$ such that 
\begin{equation}\label{eq2}
E^{n}_{(0,1)^{i}, (1,1)^{j}}(\mathscr{B}_2)= \sum_{j\in J} (x_j + d_j) + \sum_{i\in I} d_i,
\end{equation}
either $n\in I$ or $n\in J$. 

First we suppose that $n\in J$. We prove that then
$$
 E^{n}_{(0,1)^{i}, (1,1)^{j+1}}(\mathscr{B}_2) = x_n + d_n + E^{n-1}_{(0,1)^{i}, (1,1)^{j}}\big(\mathscr{B}_2\setminus \{ (x_n, d_n) \} \big).
$$
 
 A short consideration of the fact that we have three possibilities (either $x_n$ plays a role in the sum that equals $E^{n}_{(0,1)^{i}, (1,1)^{j+1}}(\mathscr{B}_2)$; or $x_n+d_n$; or neither appears in the sum) gives the following equation:
\begin{align*}
E^{n}_{(0,1)^{i}, (1,1)^{j+1}}(\mathscr{B}_2) = \max \Big(& x_n + d_n + E^{n-1}_{(0,1)^{i}, (1,1)^{j}}\big(\mathscr{B}_2\setminus \{ (x_n, d_n) \} \big),\\
&d_n + E^{n-1}_{(0,1)^{i-1}, (1,1)^{j+1}}\big(\mathscr{B}_2\setminus \{ (x_n, d_n) \} \big),\\
&E^{n-1}_{(0,1)^{i}, (1,1)^{j+1}}\big(\mathscr{B}_2\setminus \{ (x_n, d_n) \} \big)\Big).
\end{align*}
We must show that this maximum cannot be attained in the second or third expression. We prove this statement by contradiction.\\

$\bullet$ Assume that $E^{n}_{(0,1)^{i},(1,1)^{j+1}}(\mathscr{B}_2) = E^{n-1}_{(0,1)^{i}, (1,1)^{j+1}}\big(\mathscr{B}_2\setminus \{ (x_n, d_n) \} \big)$.  This means that $I_1, J_1$ exist such that 
\begin{equation}\label{eq3}
 E^{n}_{(0,1)^{i}, (1,1)^{j+1}}(\mathscr{B}_2) = \sum_{j\in J_1} (x_j + d_j) + \sum_{i\in I_1} d_i
 \end{equation}
and $n\notin I_1\cup J_1$.  Also, $s$ exists such that $s\in  J_1\cup I_1 \setminus  J\cup I$. 

Assume that $s\in J_1$. If $x_s+d_s> x_n + d_n$, then  
\[
 \sum_{j\in J\cup \{s\}\setminus \{n\}} (x_j + d_j) + \sum_{i\in I} d_i > \sum_{j\in J} (x_j + d_j) + \sum_{i\in I} d_i,
\]
which contradicts the definition of $E^{n}_{(0,1)^{i}, (1,1)^{j+1}}(\mathscr{B}_2)$.  If $x_s+d_s= x_n + d_n$, then  
\[
 \sum_{j\in J\cup \{s\}\setminus \{n\}} (x_j + d_j) + \sum_{i\in I} d_i = \sum_{j\in J} (x_j + d_j) + \sum_{i\in I} d_i= E^{n}_{(0,1)^{i}, (1,1)^{j}}(\mathscr{B}_2).
\]
The sum on the left does not involve $n$, which would imply that 
$$
E^{n}_{(0,1)^{i}, (1,1)^{j}}(\mathscr{B}_2) = E^{n-1}_{(0,1)^{i}, (1,1)^{j}}(\mathscr{B}_2\setminus\{(x_n, d_n)\}),
$$
contradicting (\ref{eq1}).  Now suppose $x_s+d_s < x_n + d_n$.  Then 
\[
 \sum_{j\in J_1\cup\{n\}\setminus\{s\}} (x_j + d_j) + \sum_{i\in I_1} d_i > \sum_{j\in J_1} (x_j + d_j) + \sum_{i\in I_1} d_i,
\]
which contradicts (\ref{eq3}).

Now assume that $s\in I_1$. 
Since $s\in I_1$ and not in $I$, and since $I$ and $I_1$ have the same cardinality, there exists $t\in I\setminus I_1$.  If $d_s=d_t$, we replace $t$ with $s$ ($s\notin I\cup J$, so this replacement is possible) and start over with $I$ replaced by $I\cup\{s\}\setminus\{t\}$, beginning with (\ref{eq2}).  The key observation is that this replacement reduces the number of elements in $I_1\setminus I$ by 1.  Thus, after at most $|I|$ such replacements, we may either assume that $I=I_1$, in which case our element $s\in J_1\cup I_1\setminus J\cup I$ necessarily lies in $J_1$ and our previous arguments prove the assertion; or we reach the situation that $I\neq I_1$, and there exist $s\in I_1\setminus(I\cup J)$ and $t\in I\setminus I_1$ with $d_s\neq d_t$.
 
If $d_s>d_t$, we arrive at a contradiction to (\ref{eq2}), since
\[
 \sum_{j\in J} (x_j + d_j) + \sum_{i\in I\cup \{s\}\setminus \{t\}} d_i > \sum_{j\in J} (x_j + d_j) + \sum_{i\in I} d_i.
\]

Now suppose that $d_s < d_t$. If $t\notin J_1$, then we replace $s$ by $t$ and arrive at a contradiction, since 
\[
 \sum_{j\in J_1} (x_j + d_j) + \sum_{i\in I_1\cup\{t\}\setminus\{s\}} d_i > \sum_{j\in J_1} (x_j + d_j) + \sum_{i\in I_1} d_i.
\]
We also must have that
\begin{equation}\label{eq6}
(x_n+d_n)+d_t \geq (x_t+d_t)+d_s.
\end{equation}
Otherwise,
\[
 \sum_{j\in J\cup\{t\}\setminus\{n\}} (x_j + d_j) + \sum_{i\in I\cup \{s\}\setminus\{t\}} d_i >  \sum_{j\in J} (x_j + d_j) + \sum_{i\in I} d_i,
 \]
which leads to a contradiction to (\ref{eq2}).  If $t\in J_1$, we use (\ref{eq6}) to obtain
\[
 \sum_{j\in J_1\cup\{n\}\setminus\{t\}} (x_j + d_j) + \sum_{i\in I_1\cup\{t\}\setminus\{s\}} d_i \geq \sum_{j\in J_1} (x_j + d_j) + \sum_{i\in I_1} d_i.
\]
Now we can replace $t$ with $n$ to obtain  
\[
E^{n}_{(0,1)^{i}, (1,1)^{j+1}}(\mathscr{B}_2) = x_n + d_n + E^{n-1}_{(0,1)^{i}, (1,1)^{j}}\big(\mathscr{B}_2\setminus \{ (x_n, d_n) \} \big).
\] 

$\bullet$ Assume that $E^{n}_{(0,1)^{i}, (1,1)^{j+1}}(\mathscr{B}_2) = d_n + E^{n-1}_{(0,1)^{i-1}, (1,1)^{j+1}}\big(\mathscr{B}_2\setminus \{ (x_n, d_n) \} \big)$.  This means that $I_1, J_1$ exist such that 
\begin{equation}\label{eq4}
 E^{n}_{(0,1)^{i}, (1,1)^{j+1}}(\mathscr{B}_2) = \sum_{j\in J_1} (x_j + d_j) + \sum_{i\in I_1} d_i + d_n
 \end{equation}
and $n\notin I_1\cup J_1$.  Take any $s\in J_1$.  Since $x_n \geq x_s$, we have $(x_n+d_n)+x_s \geq (x_s+d_s)+x_n$ and we may switch $s$ and $n$ to get
\[
\sum_{j\in J_1\cup\{n\}\setminus\{s\}} (x_j + d_j) + \sum_{i\in I_1} d_i + d_s\geq \sum_{j\in J_1} (x_j + d_j) + \sum_{i\in I_1} d_i + d_n,
\]
and because the right hand side is maximal,
\[
\sum_{j\in J_1\cup\{n\}\setminus\{s\}} (x_j + d_j) + \sum_{i\in I_1} d_i + d_s= \sum_{j\in J_1} (x_j + d_j) + \sum_{i\in I_1} d_i + d_n,
\]
and consequently,
\[
d_n + E^{n-1}_{(0,1)^{i-1}, (1,1)^{j+1}}\big(\mathscr{B}_2\setminus \{ (x_n, d_n) \} \big) =
d_n +x_n + E^{n-1}_{(0,1)^{i}, (1,1)^{j}}\big(\mathscr{B}_2\setminus \{ (x_n, d_n) \} \big).
\] 
So if $ E^{n}_{(0,1)^{i}, (1,1)^{j+1}}(\mathscr{B}_2) = d_n + E^{n-1}_{(0,1)^{i-1}, (1,1)^{j+1}}\big(\mathscr{B}_2\setminus \{ (x_n, d_n) \} \big)$, then also 
\[
E^{n}_{(0,1)^{i}, (1,1)^{j+1}}(\mathscr{B}_2) =
d_n +x_n + E^{n-1}_{(0,1)^{i}, (1,1)^{j}}\big(\mathscr{B}_2\setminus \{ (x_n, d_n) \} \big).
\]
 
We have just proved that $E^{n}_{(0,1)^{i}, (1,1)^{j+1}}(\mathscr{B}_2) = x_n + d_n + E^{n-1}_{(0,1)^{i}, (1,1)^{j}}\big(\mathscr{B}_2\setminus \{ (x_n, d_n) \} \big)$. We use this fact in the first line (\ref{line1}) of the following computation. The second line (\ref{line2}) uses the inductive hypothesis and the last line (\ref{line3}) follows from the definition of the functions we use:
\begin{align}
E^{n}_{(0,1)^{i}, (1,1)^{j+1}}(\mathscr{B}_2) &= x_n + d_n + E^{n-1}_{(0,1)^{i}, (1,1)^{j}}\big(\mathscr{B}_2\setminus \{ (x_n, d_n) \} \big), \label{line1}\\
& < x_n+d_n + E^{n-1}_{(0,1)^{i}, (1,1)^{j}}\big(\mathscr{B}_1\setminus \{ (x_n, d_n)\}\big), \label{line2}\\
&\leq E^{n}_{(0,1)^{i}, (1,1)^{j+1}}(\mathscr{B}_1). \label{line3} 
\end{align}


Now suppose that $n\in I$.  Analogously to the previous example, we will first prove that
\[
E^{n}_{(0,1)^{i+1}, (1,1)^{j}}(\mathscr{B}_2) = d_n + E^{n-1}_{(0,1)^{i}, (1,1)^{j}}\big(\mathscr{B}_2\setminus \{ (x_n, d_n)\} \big).
\]
Once again, we have $E^{n}_{(0,1)^{i+1}, (1,1)^{j}}(\mathscr{B}_2)$ is the maximum over the following three expressions:
\begin{align*}   
 E^{n}_{(0,1)^{i+1}, (1,1)^{j}}(\mathscr{B}_2) = \max \Big( &x_n + d_n + E^{n-1}_{(0,1)^{i+1}, (1,1)^{j-1}}\big(\mathscr{B}_2\setminus \{ (x_n, d_n) \} \big),\\
& d_n + E^{n-1}_{(0,1)^{i}, (1,1)^{j}}\big(\mathscr{B}_2\setminus \{ (x_n, d_n) \} \big),\\
& E^{n-1}_{(0,1)^{i+1}, (1,1)^{j}}\big(\mathscr{B}_2\setminus \{ (x_n, d_n) \} \big) \Big).
\end{align*}
We must show that this maximum cannot be attained in the first or third expression.  We prove this statement by contradiction.

Recall the following expression from (\ref{eq2}):
$ 
 E^{n}_{(0,1)^{i}, (1,1)^{j}}(\mathscr{B}_2)= \sum_{j\in J} (x_j + d_j) + \sum_{i\in I} d_i.
$
Here $n\in I$.  By our current assumption, we can write
$$
E^{n}_{(0,1)^{i+1}, (1,1)^{j}}(\mathscr{B}_2)= \sum_{j\in J_1} (x_j + d_j) + \sum_{i\in I_1} d_i,
$$
where $n\in J_1$.  Note that this only makes sense if $J_1\neq \emptyset$ (otherwise $x_n+d_n$ will not appear in the sum).  Since $n\in J_1\setminus J$, we can pick $t\in J\setminus J_1$.  Since $(x_t+d_t)+d_n \leq (x_n+d_n) + d_t,$ we can rewrite $ \sum_{j\in J} (x_j + d_j) + \sum_{i\in I} d_i$ as $\sum_{j\in J\cup\{n\}\setminus \{t\}} (x_j + d_j) + \sum_{i\in I\cup\{t\}\setminus \{n\} } d_i$ and we would be in the situation which we previously analyzed.\\
 
 $\bullet$ Assume that  $E^{n}_{(0,1)^{i+1}, (1,1)^{j}}\big(\mathscr{B}_2 \big) = E^{n-1}_{(0,1)^{i+1}, (1,1)^{j}}\big(\mathscr{B}_2\setminus \{ (x_n, d_n) \} \big)$. 
 Recall again the following expression from (\ref{eq2}):
$
 E^{n}_{(0,1)^{i}, (1,1)^{j}}(\mathscr{B}_2)= \sum_{j\in J} (x_j + d_j) + \sum_{i\in I} d_i,
$
where $n\in I$.  Additionally, we can write
$
E^{n}_{(0,1)^{i+1}, (1,1)^{j}}(\mathscr{B}_2)= \sum_{j\in J_1} (x_j + d_j) + \sum_{i\in I_1} d_i,
$
with $n\notin J_1\cup I_1$. Also, $s$ exists such that $s\in  J_1\cup I_1 \setminus  J\cup I$. 

If $d_s> d_n$, then
\[
\sum_{j\in J} (x_j + d_j) + \sum_{i\in I\cup\{s\}\setminus\{n\}} d_i> \sum_{j\in J} (x_j + d_j) + \sum_{i\in I} d_i,
\]
which is a contradiction, since the expression on the right is maximal.  If $d_s=d_n$, we may simply replace $n$ by $s$ in the sum for  $E^{n}_{(0,1)^{i}, (1,1)^{j}}(\mathscr{B}_2)$ and get that $E^{n}_{(0,1)^{i}, (1,1)^{j}}(\mathscr{B}_2)$ does not depend on $(x_n, d_n)$ and therefore
\[
E^{n}_{(0,1)^{i}, (1,1)^{j}}(\mathscr{B}_2) = E^{n-1}_{(0,1)^{i}, (1,1)^{j}}(\mathscr{B}_2\setminus\{(x_n, d_n)\}),
\]
which contradicts (\ref{eq1}).  Finally, assume that $d_s<d_n$.  If $s\in I_1$, then
\[
\sum_{j\in J_1} (x_j + d_j) + \sum_{i\in I_1\cup\{n\}\setminus\{s\}} d_i>\sum_{j\in J_1} (x_j + d_j) + \sum_{i\in I_1} d_i,
\]
which is once again a contradiction.  Now assume that $s\in J_1$. Since $x_s\leq x_n$, we have $x_s+d_s<x_n+d_n$ and therefore
\[
\sum_{j\in J_1\cup\{n\}\setminus\{s\}} (x_j + d_j) + \sum_{i\in I_1} d_i>\sum_{j\in J_1} (x_j + d_j) + \sum_{i\in I_1} d_i,
\]
which is a contradiction.


We conclude that  
\[
\begin{array}{lcl}
E^{n}_{(0,1)^{i+1}, (1,1)^{j}}(\mathscr{B}_2) &=& d_n + E^{n-1}_{(0,1)^{i}, (1,1)^{j}}\big(\mathscr{B}_2\setminus \{ (x_n, d_n)\} \big)\\
& <& d_n + E^{n-1}_{(0,1)^{i}, (1,1)^{j}} \big(\mathscr{B}_1\setminus \{ (x_n, d_n) \} \big) \\
&\leq&  E^{n}_{(0,1)^{i+1}, (1,1)^{j}}(\mathscr{B}_1). 
\end{array}
\]
We have shown that if $\mathscr{B}_1\neq \mathscr{B}_2$, we can find elementary 2-symmetric max-plus polynomials that distinguish $\mathscr{B}_1$ and $\mathscr{B}_2$. This completes the induction and the proof.
\end{proof}

\paragraph{Regularizing subsets of barcode space.}\label{regular}

In this paper, we consider subsets of barcode space defined as follows.  For a fixed $m > 0$, denote $\mathcal{B}_{\leq n}^m$ to be the subset of $\mathcal{B}_{\leq n}$ consisting of barcodes $(x_1, d_1, x_2, d_2,\ldots, x_n, d_n)$ with $d_i > 0$ such that 
\begin{equation}
\label{eqn:reg}
x_i \leq m d_i
\end{equation}
for $i=1, 2,\ldots, n$.  Given a set of finite barcodes, the value of $m$ is straightforward to determine.  While it might appear restrictive at first to consider only a subset of the barcode space, it does not actually pose any limitations in practice.  Indeed, for a finite data set generated by some finite process (e.g.,~meshes have a finite number of vertices/faces, images have limited resolution, etc.), there will only be finitely many bars, each with finite length, in the resulting barcode.  

\cite{TropCoord} gives an infinite list of functions required to distinguish any two barcodes.  Every function in this list is stable with respect to both the Wasserstein $p$- (for $p \geq 1$) and bottleneck distances.  However, in application settings, it is not possible (and quite frankly, does not make sense) to work with infinitely many functions since, in order to perform actual computations for any real data analysis, the data set in question must be finite.  Therefore, we work with subsets $\mathcal{B}_{\leq n}^m$ to ensure that we have an embedding into a finite-dimensional Euclidean space for computational feasibility.  The functions $E_{m, (1,0)^i, (1,1)^j}$ that we subsequently list in \eq{eq:trop_poly}, defined for a fixed $m$, are the finitely many functions we require to distinguish barcodes in $\mathcal{B}_{\leq n}^m$.  This restriction does not affect the Lipschitz continuity of the polynomials, which always holds, both in the regularizing subset given by (\ref{eqn:reg}) as well as in general barcode space.






\paragraph{Tropical coordinates on barcode space.}

Using Proposition \ref{minset} and restating the proof of Theorem 6.3 by~\cite{TropCoord} for the smaller family of functions, we establish that the following family of functions
\begin{equation}
\label{eq:trop_poly}
E_{m, (0, 1)^i, (1, 1)^j }(x_1, d_1, x_2, d_2, \ldots, x_n, d_n) :=E_{(0, 1)^i ,(1, 1)^j} (x_1 \oplus d_1^m, d_1, x_2 \oplus d_2^m, d_2, \ldots, x_n \oplus d_n^m, d_n)
\end{equation}
induces an injective map on $\mathcal{B}_{\leq n}^m$.  This collection of functions separates nonequivalent barcodes and are Lipschitz with respect to the Wasserstein $p$- and bottleneck distances \citep{TropCoord}.  We now give an example (with steps) for calculating these functions and evaluating them on barcodes. 

\begin{example}\label{example}
Fix $n = 2$. The set of orbits under the $S_2$ action is
$$
\mathscr E_2 / S_2 = \left\{ \begin{array}{c} \bigg[ \begin{pmatrix} 1 & 1\\ 1 & 1 \end{pmatrix} \bigg], \bigg[ \begin{pmatrix} 1 & 0\\ 1 & 1 \end{pmatrix} \bigg], \bigg[ \begin{pmatrix} 1 & 1\\ 0 & 1 \end{pmatrix} \bigg], \bigg[ \begin{pmatrix} 0 & 0\\ 1 & 1 \end{pmatrix} \bigg],\\ ~\\
\bigg[ \begin{pmatrix} 1 & 0\\ 1 & 0 \end{pmatrix} \bigg],
\bigg[ \begin{pmatrix} 1 & 0\\ 0 & 1 \end{pmatrix} \bigg],
\bigg[ \begin{pmatrix} 0 & 1\\ 0 & 1 \end{pmatrix} \bigg],
\bigg[ \begin{pmatrix} 0 & 1\\ 0 & 0 \end{pmatrix} \bigg],
\bigg[ \begin{pmatrix} 1 & 0\\ 0 & 0 \end{pmatrix} \bigg] \end{array} \right\}.
$$
According to Proposition~\ref{minset}, we need only consider the functions defined by the following orbits to separate barcodes when $n=2$:
\begin{align}
\label{orbits_ex}
\bigg[ \begin{pmatrix} 0 & 1\\ 0 & 0 \end{pmatrix} \bigg], \bigg[ \begin{pmatrix} 0 & 1\\ 0 & 1 \end{pmatrix} \bigg],  \bigg[ \begin{pmatrix} 0 & 0\\ 1 & 1 \end{pmatrix} \bigg], \bigg[ \begin{pmatrix} 1 & 1\\ 0 & 1 \end{pmatrix} \bigg], \bigg[ \begin{pmatrix} 1 & 1\\ 1 & 1 \end{pmatrix} \bigg].
\end{align}
Suppose that we have two barcodes $\mathscr B_1 = \big\{ (1,2), (3,1) \big\}$ and $\mathscr B_2 = \big\{ (2,2) \big\}$, where $\mathscr B_1, \mathscr B_2 \in \mathcal B_{\leq 2}.$
\begin{enumerate}[1.]
\item Compute $m$: For intervals $(1,2),\, (3,1),\, (2,2)$, find the smallest $m$ such that $x_i \leq m d_i$ for all $i$.  This $m$ is determined by rearranging (\ref{eqn:reg}) as $x_i \leq md_i \Leftrightarrow \frac{x_i}{d_i} \leq m$, and must be greater than the quotients $\frac{x_i}{d_i}$ for $i=1,2,3$.  These quotients are $\frac{1}{2}, \frac{3}{1}, 1$, and therefore we take $m = 3$.  Having determined $m$, it follows that $\mathscr B_1, \mathscr B_2 \in \mathcal B_{\leq 2}^3$.

\item Determine the 2-symmetric max-plus polynomials from the subcollection of orbits \eq{orbits_ex}:
\begin{align*}
E_{3,(0,1)}(x_1, d_1, x_2, d_2) &= E_{3,(0,1),(0,0)}(x_1, d_1, x_2, d_2)\\
&= d_1 \boxplus d_2\nonumber\\
&= \max\{d_1, d_2 \},\nonumber\\
E_{3,(0,1)^2}(x_1, d_1, x_2, d_2) &= E_{3,(0,1),(0,1)}(x_1, d_1, x_2, d_2)\\
&= d_1 d_2\nonumber\\
&= d_1 + d_2,\nonumber\\
E_{3,(1,1)}(x_1, d_1, x_2, d_2) &= E_{3,(0,0),(1,1)}(x_1, d_1, x_2, d_2)\\
&= (x_2 \oplus d_2^3)d_2 \boxplus (x_1 \oplus d_1^3)d_1\nonumber\\
&= \max\big\{ \min\{x_2, 3d_2 \} + d_2, \min\{ x_1, 3d_1 \} + d_1 \big\},\nonumber\\
E_{3,(0,1),(1,1)}(x_1, d_1, x_2, d_2) &= (x_1 \oplus d_1^3 )d_1d_2 \boxplus (x_2 \oplus d_2^3) d_2d_1\nonumber\\
&= \max\big\{\min\{ x_1, 3d_1 \} + d_1 + d_2, \min \{ x_2, 3d_2 \} + d_2 + d_1 \big\},\nonumber\\
E_{3,(1,1)^2}(x_1, d_1, x_2, d_2) &= E_{3,(1,1),(1,1)}(x_1, d_1, x_2, d_2)\\
&= (x_1 \oplus d_1^3)d_1(x_2 \oplus d_2^3)d_2\nonumber\\
&= \min \{ x_1, 3d_1 \} + d_1 + \min\{x_2, 3d_2 \} + d_2.\nonumber
\end{align*}

\item Evaluate on $\mathscr B_1$ and $\mathscr B_2$.


\end{enumerate}
The Euclidean-space vector representation of $\mathscr B_1$ and $\mathscr B_2$ is $(2, 3, 4, 6, 7)$ and $(2, 2, 4, 4, 4)$, respectively.
\end{example}


\section{Sufficient statistics and likelihoods for barcodes}
\label{sec:suff}

In this section, we motivate the need for sufficient statistics for persistent homology.  We also present our main result that tropical functions, defined according to the max-plus and min-plus semirings and listed in (\ref{eq:trop_poly}), are sufficient statistics for persistence barcodes.  Lastly, we discuss the implications of our result on statistical parametric inference by imposing the structure of exponential family distributions and defining likelihood functions for the tropical coordinatized representation of barcodes.


\subsection{Sufficient statistics and persistent homology}

Although a well-defined metric space, the space of barcodes exhibits a complex geometry with respect to the $L_2$-Wasserstein distance, it is an Alexandrov space with curvature bounded from below \citep{Turner2014}.  In the more conventional sense, this structure is theoretically not restrictive since under general Wasserstein $p$-distances, it is also a Polish space and admits well-defined characterizations of expectations, variances, and conditional probabilities \citep{0266-5611-27-12-124007}.  This means that formal probability measures and distributions may be formulated for persistence diagrams and barcodes.  Probability measures on the space of persistence diagrams have been of key interest for conducting statistical inference in TDA \citep[e.g.,][]{adler2011topological, carlsson2014, Blumberg2014}.  However, formulating explicit, parametric probability distributions directly on the space of persistence diagrams---under either the Wasserstein $p$-distance or bottleneck distance as the limiting case of the Wasserstein distance---remains a significant challenge. Although nonparametric methods for persistent homology exist \citep[e.g.,][]{Chazal:2014:SCP:2582112.2582128, Robinson2017}, in classical statistics, parametric modeling is the canonical starting point for data analysis and inference.  Fitting nonparametric models can be more computationally intensive than parametric methods, which is another drawback given that computing persistent homology does not scale well for very large data sets.  In particular, the upper size bound for Vietoris--Rips complexes based on $N$ vertices is $2^{\mathcal{O}(N)}$, and the computational complexity for persistent homology implemented by the standard algorithm \citep{elz-tps-02, ZC} is cubic in the number of simplices \citep[for a detailed survey on computing persistent homology, see][]{Otter2017}.  Furthermore, the existence of Fr\'echet means has been proven, but these are not unique \citep{Turner2014}.  It is therefore difficult to perform statistical analyses, particularly parametric inference, directly on the space of persistence diagrams.


\paragraph{Statistical sufficiency.}

The difficulty of defining explicit probability distributions directly on the space of barcodes is our main motivation for the present work.  To solve this problem, our strategy is to achieve sufficiency for barcodes via Euclidean coordinatization.  This way we can establish distributional forms for the coordinatized equivalents of barcodes, as long as our mapping is a sufficient statistic.  Given that the coordinatization is a topological embedding and the mappings are continuous and injective between barcode and Euclidean space, any parametric inference performed on these coordinatized equivalents naturally preserves all key properties in barcode space.  The property of injectivity also ensures the existence of an inversion.  Therefore, the results of parametric analyses on Euclidean space can theoretically be relayed back into barcode space by the inverse mappings.  Note, however, that this is not necessary because sufficiency guarantees that (for analytical purposes) all of the information from the original barcodes is retained by their coordinates.

Intuitively, a statistic is said to be {\em sufficient} for a family of probability distributions if the original sample from which the statistic was computed provides no more additional information on the underlying probability distribution, than does the statistic itself.  A subtlety to remark here is that the term ``statistic'' has a dual meaning: the same term is used to refer to the point estimate, i.e., evaluation of a function, as well as to the function itself (as am example, the mean is often used to refer to both the value obtained from computing the empirical average of a finite set of data points as well as the function itself, which is the sum of data values divided by the number of data points).  The concept of sufficiency retaining all information applies to both the statistic as a function of the data mapping to a codomain as well as the statistic as an evaluation of such a function.  In the former case, sufficiency refers to a property of the statistic with respect to probability measures on the corresponding domain and codomain, where the exact distribution may be unknown but the spaces nevertheless admit well-defined probability measures.  In the latter case, sufficiency refers to the purpose of determining the exact probability distribution that generated the observed data, which is a key aim of parametric inference.  In this case, knowing a sufficient statistic contains all of the information needed to compute any estimator of the parameter for the assumed underlying distribution.

\setcounter{theorem}{6}
\begin{definition}
Let $X$ be a vector of observations of size $n$ whose components $X_i$ are independent and identically distributed ({\em i.i.d.}).  Let $\vartheta$ be the parameter that characterizes the underlying probability distribution that generates $X_i$.  A statistic $T(X)$ is {\em sufficient} for the parameter $\vartheta$ if the conditional probability of observing $X$, given $T(X)$, is independent of $\vartheta$.  Equivalently,
\begin{equation*}
\mathbb{P}(X = x \mid T(X) = t, \vartheta) = \mathbb{P}(X = x \mid T(X) = t).
\end{equation*}
\end{definition}
\noindent The above definition applies to point evaluations of the statistic $T(X)$.  From the measure-theoretic perspective, a statistic $T(\cdot)$ is sufficient for a given set $\mathcal M$ of probability measures on a $\sigma$-algebra $\mathbb{S}$ if the conditional probability of some subset $E \subseteq \mathbf{S}$, given a value $y$ of $T(\cdot)$, is invariant for every probability measure in $\mu \in \mathcal M$. 

We now recall an important result in the mathematical statistics literature for assessing statistical sufficiency: the factorization criterion for sufficient statistics \citep{Fisher309, Neyman}. The following theorem is classically used as a definition for checking sufficiency, in particular, for statistics that are point estimates.

\setcounter{theorem}{0}
\begin{theorem}[Fisher--Neyman factorization criterion]
\label{thm:fisher_neyman}
If the probability density function for the observed data is $f(x; \vartheta)$, then the statistic $T = T(x)$ is sufficient for $\vartheta \in \Theta$ if and only if nonnegative functions $g$ and $h$ exist such that
\begin{equation}
f(x; \vartheta) = h(x)g\big( T(x) ; \vartheta \big).\label{eqn:fisher_neyman}
\end{equation}
Namely, the density $f(\cdot)$ can be factored into a product where one factor $h(\cdot)$ is independent of $\vartheta$, and the other factor $g(\cdot)$ depends on $\vartheta$ and only depends on $x$ only through $T(\cdot)$.
\end{theorem}
This factorization need not be unique for any given distribution.  For example, to see that the identity function $T(X) = X$ is a sufficient statistic, we may factorize by setting $h(x) = 1$, and then $f(x; \vartheta)= g\big(T(x); \vartheta \big)$ for all $x$.  This is the trivial sufficient statistic and is interpreted as the data themselves being sufficient for the data.

\begin{example}[normal distribution: unknown mean $\mu$, known variance $\sigma^2$]
We denote the known parameter by the subscript $\sigma$.  The normal probability density is
$$
f_\sigma(x; \mu) = \frac{1}{\sqrt{2\pi\sigma^2}}\exp\bigg\{ - \frac{(x - \mu)^2}{2\sigma^2} \bigg\}.
$$
By setting
$$
h_\sigma(x) := \frac{1}{\sqrt{2\pi\sigma^2}}\bigg\{ - \frac{x^2}{2\sigma^2} \bigg\},\mbox{~~~}
T_\sigma(x) := \frac{x}{\sigma},\mbox{~~~}
g_\sigma(y) := \exp\bigg\{ \frac{\mu}{\sigma} \cdot y - \frac{\mu^2}{2\sigma^2} \bigg\}
$$
and substituting in (\ref{eqn:fisher_neyman}) with $y = T_\sigma(x)$, we get
$$
f_\sigma(x;\mu) = \frac{1}{\sqrt{2\pi\sigma^2}}\exp\bigg\{ - \frac{x^2}{2\sigma^2} \bigg\} \cdot \exp\bigg\{ \frac{\mu}{\sigma} \cdot \frac{x}{\sigma} - \frac{\mu^2}{2\sigma^2} \bigg\}
= \frac{1}{\sqrt{2\pi\sigma^2}}\exp\bigg\{ -\frac{(x-\mu)^2}{2\sigma^2} \bigg\}
$$
as desired.

\end{example}

Theorem \ref{thm:fisher_neyman} is valid for arbitrary families of distributions when studying the sufficiency of point estimates.  The following measure-theoretic reinterpretation asserts the generality of the theorem, particularly in the case where sufficiency refers to statistics on probability measures \citep{halmos1949}.  Essentially, the result depends on the assumption that the family of probability measures $\mathcal P_\vartheta$ is dominated by a single $\sigma$-finite measure.  In the case of discrete variables, this amounts to the existence of a countable set $S$, which does not depend on $\vartheta$, for which $\mathcal{P}_\vartheta(S) = 1$ for all $\vartheta \in \Theta$.  In the case of absolutely continuous variables, this amounts to the measurability of the statistic.

\begin{theorem}[\cite{halmos1949}]
\label{thm:halmos_savage}
A necessary and sufficient condition that the statistic $T(\cdot)$ be sufficient for a dominated set $\mathcal M$ of measures on a $\sigma$-algebra $\mathbb{S}$ is that for every $\mu \in \mathcal M$, the density $f_\mu := \diff\mu/\diff\lambda$ (where $\diff\mu/\diff\lambda$ is the Radon--Nikodym derivative) admits the factorization
$$
f_\mu(x) = h(x)g_\mu\big( T(x) \big),
$$
where $h(\cdot) \not= 0$ is a measurable function with respect to $\mathbb{S}$, and $g_\mu(\cdot)$ is a measurable function with respect to the codomain $\mathcal T$ of $T(\cdot)$.
\end{theorem}

As will be shown subsequently, in our case where the statistic is given by tropical polynomials in \eq{eq:trop_poly}, $\mathbf{S}$ corresponds to the space of persistence barcodes $\mathcal B_{\leq n}^m$, while $\mathcal T$ is Euclidean space.  Although the above theorem holds in broader generality, the case of absolutely continuous variables is sufficient for our setting since the tropical polynomials are Lipschitz continuous and are therefore absolutely continuous as well.


\paragraph{Sufficient statistics for persistence barcodes.}

Given the existence of parametric distributions and probability measures on barcode space \citep{0266-5611-27-12-124007}, our main theorem ascertains that the map from $\mathcal{B}_{\leq n}^m$, as an embedding into Euclidean space via tropical functions, is a sufficient statistic.  Sufficiency for statistics is classically asserted by the injectivity of continuous mappings \cite[e.g.,][]{halmos1949, browder1992mathematics}.  We follow this same approach in our proof.

\begin{theorem}
\label{sufficient}
Consider the subset $\mathcal{B}_{\leq n}^m$ of $\mathcal{B}_{\leq n}$ consisting of barcodes $[(x_1, d_1, x_2, d_2,\ldots, x_n, d_n)]$ with $d_i > 0$ such that $x_i \leq m d_i$ for some fixed $m$ and $i=1, 2, \ldots, n$.  Let $\mathscr P_\Theta$ be the family of probability measures on $\mathcal{B}_{\leq n}^m$.  The map defined by
\begin{align}
T: \mathcal{B}_{\leq n}^m & \rightarrow \mathbb{R}^d \nonumber\\
\mathscr{B} & \mapsto \big( E_{m, (0, 1)^i, (1, 1)^j }( x_1, d_1, x_2, d_2, \ldots, x_n, d_n) \big)_{i +j \in \mathbf{N}_{\leq n}}(\mathscr{B}) \label{eqn:suff_stat},
\end{align}
where $d$ is the number of orbits used to define separating functions, is a sufficient statistic for all $P_\vartheta \in \mathscr{P}_\Theta$.
\end{theorem}

\begin{proof}
Denote the image of $\mathcal{B}_{\leq n}^m$ under $\big( E_{m, (0, 1)^i, (1, 1)^j }(x_1, d_1, x_2, d_2, \ldots, x_n, d_n) \big)_{i +j \in \mathbf{N}_{\leq n}}$ by $\mathcal{T}\subseteq \mathbb{R}^d$.  On $\mathcal{B}_{\leq n}^m$, we take the bottleneck distance, and on $\mathcal{T}$, we take the standard Euclidean distance.  We can view $\mathcal{B}_{\leq n}^m$ as a measurable space by taking the Borel $\sigma$-algebra to be the smallest $\sigma$-algebra that contains all open sets in the metric space $(\mathcal{B}_{\leq n}^m, d_{\infty}^B)$ with topology generated by open balls.  The same goes for $\mathcal{T}$ with standard Euclidean distance \citep[Chapter 1]{1967}.  Both metric spaces are separable \citep{0266-5611-27-12-124007} and complete ($\mathcal{B}_{\leq n}^m$ is complete since it is a closed subset of the barcode space).



To apply Theorem~\ref{thm:halmos_savage} and show that $T$ generates sufficient statistics for barcodes on $\mathcal{B}_{\leq n}^m$, we use the property that the map $T$ is injective, proven in Proposition~\ref{minset}.  By definition, this means that there exists some $\varphi \colon \mathcal{T} \to \mathcal{B}_{\leq n}^m$ such that $\varphi \circ T = \textrm{Id}_{\mathcal{B}_{\leq n}^m}$ and $T \circ \varphi = \textrm{Id}_{\mathcal{T}}$.  For notational simplicity, we denote $f_\vartheta(\mathscr B) := f(\mathscr B; \vartheta)$ and $g_\vartheta(T(\mathscr B)) := g(T(\mathscr B); \vartheta)$.  The condition of Theorem~\ref{thm:halmos_savage} holds for $T$ where $g_\vartheta = f_\vartheta \circ \varphi$:
\[
\begin{array}{lcl}
f_\vartheta (\mathscr B)  &= & f_\vartheta \big(\varphi\big(T(\mathscr B) \big)\big), \\
&=& f_\vartheta \circ \varphi\big(T(\mathscr B) \big),\\ 
&=& g_\vartheta\big(T(\mathscr B) \big).
\end{array}
\]
\noindent We now proceed to verify that all the relevant functions are measurable, i.e., that for each of them, the preimage of any measurable set is measurable.  The function $h(\mathscr B) \equiv 1$ is a constant and therefore continuous, from where it follows that it is measurable \citep[Theorem 1.5 in Chapter 1]{1967}.  To show that $f_\vartheta$ is measurable, notice that $g_\vartheta = f_\vartheta \circ \varphi$.  Since $g_\vartheta$ is measurable by assumption (Theorem \ref{thm:halmos_savage}), it will be sufficient to show $\varphi \colon \mathcal{T} \to \mathcal{B}_{\leq n}^m$ is measurable.  Since the map $T$ is Lipschitz continuous with respect to the bottleneck, Wasserstein $p$-, and the standard Euclidean distances, it is measurable with respect to Borel $\sigma$-algebras as specified previously \citep[Theorem 1.5 in Chapter 1]{1967}.  By the Kuratowski theorem \citep[Corollary 3.3 in Chapter 3]{1967}, the inverse of an injective, measurable map between separable metric spaces equipped with Borel $\sigma$-algebras is measurable, which leads to the conclusion that $\varphi$ is Borel measurable.

This proves the result when $\mathcal{B}_{\leq n}^m$ is equipped with the bottleneck distance.  A verbatim proof provides the same result for all Wasserstein-$p$ distances.
\end{proof}

As previously mentioned, the dimension $d$ of the codomain $\mathbb{R}^d$ for the tropical mappings in \eq{eqn:suff_stat} need not equal the number of all orbits given by the symmetric group action, as a result of Proposition~\ref{minset}.  The dimension $d$ is determined by the restriction that $i + j \in \mathbb{N}_{\leq n}$, where $n$ is the number of bars in the largest barcode, which simplifies to $2d = 2n + n(n+1)$.

\subsection{Exponential families and likelihood functions for persistent homology}


Sufficient statistics are strictly parametric by nature and thus are most relevant in studies centered around inference.  Parametric inference is carried out under the assumption that the parameter space is necessarily finite and therefore bounded with respect to sample size.  The only class of distributions that satisfies this restriction is known as the {\em exponential family} \citep{darmois,koopman,pitman1936}, generated by a particular case of Theorem~\ref{thm:fisher_neyman}.

For many applications, data can be appropriately modeled using exponential family distributions.  This is especially true for biological applications, which is the focus of the case study presented in Section \ref{sec:app}.  Here, Gaussian distributions are commonly assumed for many quantitative traits such as human height \citep[e.g.,][]{Turchin:2012aa,Yang:2012aa} and weight \citep[e.g.,][]{Willer:2009aa}; while dichotomous phenotypes, such as disease status, are usually modeled as Bernoulli random variables \citep[e.g.,][]{Consortium:2007aa}.  Counts, such as gene expression data, are often modeled as Poisson or negative binomial distributions \citep[e.g.,][]{Anders:2010aa,Love:2016aa}.  These distributions are all members of the exponential family; indeed, many more important distributions, such as the exponential, geometric, gamma, beta, chi-squared, and Dirichlet distributions are also members of the exponential family.  The application of such distributions also extend beyond biology and arise in many other scientific disciplines (e.g., quality control).

\setcounter{theorem}{7}
\begin{definition}
Let the following restriction be applied to \eq{eqn:fisher_neyman} in Theorem~\ref{thm:fisher_neyman},
\begin{equation}
\label{eqn:exp_fam}
g( T(x) ; \vartheta) = \exp\big\{ \langle \eta(\vartheta),  T(x) \rangle - A(\vartheta) \big\},
\end{equation}
where $\langle \cdot, \cdot \rangle$ denotes the usual inner product in Euclidean space, and $\eta(\vartheta)$, $T(x)$, and $A(\vartheta)$ are known functions.  The resulting class of probability distributions spanned by \eq{eqn:fisher_neyman}, with $g(\bullet)$ given by (\ref{eqn:exp_fam}), is known as the {\em exponential family} of probability densities.
\end{definition}

Though less general than the form in (\ref{eqn:fisher_neyman}), the exponential family is nevertheless an important class of distributions that is comprised of many of the most common distributions in statistics.  Given an observation $x$, the {\em likelihood function} $\mathscr L(\cdot)$ is a function of the parameter $\vartheta$ from the probability distribution $f_\vartheta$ that generates the observation within the sample.  Namely,
\begin{align}
\mathscr{L}(\vartheta \mid x) & = f_\vartheta(x) \nonumber\\
& = h(x)g( T(x); \vartheta), \label{eqn:mle_suff}
\end{align}
for any given sufficient statistic.  For an exponential family model, if $g(\cdot)$ is given by (\ref{eqn:exp_fam}) and we denote $a(\vartheta) := \exp\{-A(\vartheta)\}$, then we obtain
\begin{equation}
\label{eqn:mle_exp}
\mathscr{L}(\vartheta \mid x) = h(x)a(\vartheta)\exp\big\{ \langle \eta(\vartheta), T(x) \rangle \big\}.
\end{equation}


\paragraph{Joint likelihoods for persistence barcodes.}

We now extend the above formulation to the multivariate case and, in particular, we consider the case for multiple barcodes.  In this setting, we denote collections of $K$ barcodes by $\mathscrbf{B} = (\mathscr B_1, \mathscr B_2, \ldots, \mathscr B_K)$, the tropicalized versions of these collections by $T(\mathscrbf{B})$, and the collection of distribution parameters by $\bm{\vartheta} = (\vartheta_1,\ldots,\vartheta_K)$. We will use $\mathscr{B}$, $T(\mathscr{B})$, and $\vartheta$ to notate singular (or univariate) elements of these vectors, respectively.  For $K$ barcodes marginally distributed according to $f_{\bm{\vartheta}}$ as defined by \eq{eqn:mle_suff}, the joint likelihood for the $K$ observations may be written as
\begin{align}
\label{eqn:mle_trop_suff}
\mathscr{L}(\bm{\vartheta} \mid \mathscrbf{B}) & = f_{\bm{\vartheta}}(\mathscrbf{B}) \nonumber\\
& = h(\mathscrbf{B})g(T(\mathscrbf{B}); \bm{\vartheta}),
\end{align}
with expectation $\mathbb{E}[\mathscrbf{B}]) = \bm{\mu}$ and some general covariance structure $\mathbb{V}[\mathscrbf{B}] = \bm{\Sigma}$ between the barcodes.  When $f_{\bm{\vartheta}}$ belongs to the exponential family, as denoted in (\ref{eqn:mle_exp}), the joint likelihood for the $K$ observations may be written as the following:
\begin{equation}
\mathscr{L}(\bm{\vartheta} \mid \mathscrbf{B}) = a(\bm{\vartheta})^K \exp\big\{ \big\langle \eta(\bm{\vartheta}), T(\mathscrbf{B}) \big\rangle \big\}.
\label{eqn:mle_trop_exp}
\end{equation}
Since the form of our tropical sufficient statistic (\ref{eqn:suff_stat}) is a vector in Euclidean space for each barcode, this quantity is well defined.  Most importantly, the inner product is readily computable in Euclidean space with respect to the natural Euclidean metric (as opposed to the space of barcodes which does not yield a natural definition for an inner product due to its complex geometry, as discussed above).



\section{Application: Studying evolutionary behavior in viral data sets}
\label{sec:app}

We now demonstrate the utility of our sufficiency result in applications regarding two infectious viral diseases: HIV and avian influenza.  More specifically, we study both dimensions 0 and 1 of persistent homology.  In dimension 0, we concretely demonstrate sufficiency in the preservation of biological characteristics of HIV.  In dimension 1, we impose the parametric structure of an exponential distributive family to carry out pairwise comparisons between the marginal distributions of intra- and inter-subtype reassortment in avian influenza.

\subsection{Demonstrating Sufficiency in HIV Transmission Clustering}
\label{subsec:hiv}

HIV transmission clusters represent groups of epidemiologically-related individuals who share a common recent viral ancestor \citep{Hue22032005, pmid28353537}.  Recently, \cite{juan1} reported the existence of a prominent HIV transmission cluster affecting more than 100 people from population of men who have sex with men (MSM) in Valencia, Spain. 

Since epidemiologically related groups are often characterized by the identification of small pairwise genetic distances, sequences derived from a transmission cluster can be easily differentiated from those derived from unrelated patients by statistically assessing their clustering behavior (see Figure \ref{Fig1}).  Using persistent homology, dimension 0 ($\text{PH}_0$) captures the clustering of sequences based on vertical evolution \citep[i.e., the amount of evolution explained solely by mutation][]{Chan12112013}.  This is shown via bottleneck distances in Figure \ref{Fig2}.  This particular example serves as a real-data example to demonstrate that tropical coordinatization retains all information contained in barcodes.  Specifically, we show that no recombination events occurred in this specific HIV data set, indicated by the similar clustering pattern shown among the phylogenetic trees obtained for each independent segment of analyzed viral samples.  More generally, persistent homology not only provides information on clusters that would also be recovered by single-linkage clustering in dimension $0$, but it also provides relevant information on the potential for higher dimensional persistence intervals to appear.  Higher dimensional persistence intervals would suggest that there exists recombination or reassortment events.  Such information cannot be obtained from single-linkage clustering alone.


In order to demonstrate sufficiency and verify that the tropical representations of the barcodes retain the same biologically-relevant information in dimension 0, we generated 10 random subsets of 30 polymerase sequences from (i) the MSM transmission cluster, (ii) the reference data set (i.e., the unrelated patients), and (iii) sequences from both the transmission cluster and the reference dataset.  In total 10 $\times$ 30 sequences were analyzed.  Matrices of pairwise genetic distances were obtained from each subset.  The entries of these matrices are evolutionary distances between sequences, calculated from observed differences between sequences using evolutionary models (see Appendix \ref{appendix:data} for further details).

Obtaining phylogenetic trees using distance-based methods as described above can be viewed analogously to embedding a general metric space (which in our case, is spanned by distances between gene loci) into tree metrics.  This is a specific case of a more general research problem in the theory of finite metric spaces \citep{ABN07tree}.  As in previous work \citep{Chan12112013}, the underlying topological space we consider is the phylogenetic tree generated by the genetic distance matrix.  A phylogenetic tree is a specific subgraph defined by a set of vertices called {\em leaves} and a set of edges called {\em branches}---each with positive length, representing evolutionary time.  A leaf is a vertex with degree 1, a root is a vertex with degree 2, and nonleaf vertices have degree of at least 2.  The graph in consideration is the topological space arising from a usual graph $G = (V,E)$, where $V$ denotes the set of vertices and $E$ is the set of edges.  A graph as a topological space can be obtained from a usual graph by replacing the vertices by points and the edges $e = st$ by a copy of the unit interval $[0,1]$, where 0 identifies with the point associated to the vertex $s$ and 1 identifies with the point associated to the vertex $t$.  As a topological space, a graph is the simplicial $1$-complex generated by $G$ \citep{hatcher2002algebraic}. Further details in mathematical phylogenetics from the perspective of algebraic statistics can be found in \cite{pachter2005algebraic}, and from the perspective of topology and metric geometry in \cite{lin2018tropical}.  Since the subgraph of a graph as topological space, is itself a topological space, a phylogenetic tree can be viewed as the underlying topological space \citep{Chan12112013}.  In our setting, leaves in a phylogenetic tree are viral sequences sampled from individuals, and edges are genetic distances between them.


We obtained the dimension 0 Vietoris--Rips persistence barcodes using Ripser \citep{ripser} and evaluated them via the tropical functions described in Section \ref{subsec:coords}.  The resulting barcodes (and their post-transformed Euclidean representations) from subsets (i) and (ii) provide information on intragroup distances, while those from subset (iii) provide information on intergroup distances (i.e., transmission cluster versus unrelated patients).  Figures \ref{Fig1} and \ref{newFig2} show that the three subsets can be easily differentiated, thus demonstrating that the tropical coordinatization of the barcodes are sufficient statistics for epidemiologically related individuals. 

\subsection{Analyzing intra- and intersubtype reassortment in avian influenza}
\label{subsec:bird_flu}

The influenza virus presents a genome that consists of eight different ribonucleic acid (RNA) molecules (segments).  Consequently, genetic reassortment---that is, the process of swapping gene segments---is an important evolutionary process that redistributes genetic variability.  Subtype classification of influenza is based on the analysis of two surface proteins: hemagglutinin (HA) and neuraminidase (NA).  There exist 18 types of HA (indexed from 1 to 18) and 11 types of NA (indexed from 1 to 11).  Gene reassortment can occur either between viruses from the same subtype \citep[i.e., intrasubtype reassortment][]{pmid28034300}, or between viruses from different subtypes \citep[i.e., intersubtype reassortment][]{10.1371/journal.pone.0149608}.  From a broad perspective, detecting gene reassortment is key to understanding mutations within the evolutionary dynamics of viruses.

Traditionally, the detection of reassortment events is based on molecular phylogenetic analyses, in which different genomic regions (i.e., segments) are analyzed independently.  Specific details on this procedure given in Appendix~\ref{appendix:trees}.  As mentioned previously, if no reassortment has occurred, phylogenetic trees obtained for each independent segment show a similar clustering pattern of the viral samples that are analyzed.  This indicates that all segments share the same evolutionary history.  Alternatively, when a reassortment event occurs, it is detected by a segment generating a phylogenetic tree with a differing clustering pattern.  This indicates the existence of segment swapping \citep{PEREZLOSADA2015296}.  The detection and quantification of reassortment events with phylogenetic approaches is a slow and tedious process, partly due to the fact that obtaining accurate phylogenetic trees can take days, or even weeks, when hundreds or thousands of samples are analyzed \citep{pmid28034300}.  The use of persistent homology in detecting genetic reassortment events poses a clear advantage over phylogenetic analysis, not only due to the increased computational speed, but also because of the enhanced interpretability: dimension 1 persistence intervals provide explicit information on the genetic divergence of the sequences involved in the reassortment event \citep{Chan12112013,1406.4582}.

Given that viruses from the same subtype are more closely related to each other than they are to viruses from different subtypes, it is expected that lengths of intervals in $\textrm{PH}_1$ barcodes derived from events of intrasubtype reassortment will be shorter than those from intersubtype reassortment \citep{1406.4582}.  We will demonstrate this parametrically using our sufficiency result.  In this context, sufficiency means preserving the information that intra- and intersubtype samples should not cluster together.

Since we are interested in analyzing reassortment events in this example, we focus on avian influenza type A and study dimension 1 persistence.  As discussed above, in the same manner that phylogenetic trees are specific subgraphs and therefore are topological spaces, here the subgraph representing reassortment events is a directed acyclic graph and is also a topological space \citep{Chan12112013}.  By uniform sampling without replacement, as above, we generated 100 random subsets of 56 concatenated HA and NA sequences representing: (i) only H5N1 sequences, and (ii) sequences from all subtypes $\textrm{H5N}u$ and $\textrm{H}v\textrm{N1}$ where $u = 1, 2, \ldots, 12$ and $v = 1, 2, \ldots, 9$, respectively.  In total, 100 $\times$ 56 sequences were analyzed.  H5N1 is of particular interest because of its characteristics as a highly pathogenic subtype of avian influenza with over 60\% mortality and its high potential to cause a pandemic affecting the human population \citep{Kilpatrick19368}.  The resulting barcodes from subsets (i) provide information on intrasubtype reassortment events, while those from subsets (ii) provided information on intersubtype reassortment events. 


Sufficiency is an inherently parametric construct, and since we have sufficient statistics for the family of probability measures on the space of barcodes, the parametric structure of probability distributions between the two spaces is preserved \citep{10.2307/2238284, billingsley2012probability}.  We assume that the Euclidean representation of barcodes, from each of the intra- and intersubtype populations, jointly come from a multivariate normal distribution as in \eq{eqn:mle_trop_exp}:
$$
T(\mathscrbf{B}) \sim \mathcal{N}(\bm{\mu},\bm{\Sigma}),
$$
with $K$-dimensional mean vector $\bm{\mu}$, and $K \times K$ symmetric and positive semidefinite covariance matrix $\bm{\Sigma}$.  This implies that we may assume that each barcode to have a marginal univariate distribution:
\begin{align}
\label{eqn:marginal_normal}
T(\mathscr B_k)\sim \mathcal{N}(\mu_k,\sigma^2_k), \quad k = 1, 2, \ldots,K,
\end{align}
with scalar mean and variance $\vartheta = \{\mu_k,\sigma^2_k\}$, and $K = \{43,100\}$ for the intra- and intersubtype reassortment events, respectively.  The data on intra- and intersubtype reassortment events were obtained from the publicly available data source (see Appendix \ref{appendix:data} for details); data on intrasubtype reassortment are more limited than those on intersubtype reassortment.  The distribution plots for each class of reassortment events are shown in Figure \ref{Fig4}.

In this parametric setting, we may study pairwise comparisons between the marginal distributions of each $T(\mathscr B_k)$ in order to assess the following clustering behavior: samples that cluster together will tend to have marginal distributions that are more similar or, in other words, have shorter distances.  A function commonly used to measure differences between two probability distributions is an $f$-divergence.  Two particular instances of an $f$-divergence are the Kullback--Leibler divergence \citep{Kullback:1951aa} and Hellinger distance \citep{zbMATH02637393}, which we now briefly define.

\begin{definition}
Assume that $T(\mathscr B_i)$ and $T(\mathscr B_j)$ are probability measures that are absolutely continuous with respect to a third probability measure, $\lambda$.  The {\em Kullback--Leibler divergence} from $T(\mathscr B_i)$ to $T(\mathscr B_j)$ is defined to be the integral
\begin{align}
\label{eqn:kld}
\mbox{KLD}\big(T(\mathscr B_i)\,\|\,T(\mathscr B_j) \big) = \int \frac{\diff T(\mathscr B_i)}{\diff\lambda}\log\left(\frac{\diff T(\mathscr B_i)}{\diff\lambda} \Big/ \frac{\diff T(\mathscr B_j)}{\diff\lambda}\right)\diff\lambda.
\end{align}
Alternatively, the {\em Hellinger distance} between the measures may be computed by solving the following
\begin{align}
\label{eqn:hellinger}
\mbox{H}^2\big(T(\mathscr B_i), T(\mathscr B_j) \big) = \frac{1}{2}\int\left(\sqrt{\frac{\diff T(\mathscr B_i)}{\diff\lambda}}-\sqrt{\frac{\diff T(\mathscr B_j)}{\diff\lambda}}\right)^2\diff\lambda,
\end{align}
where $\diff T(\mathscr B_i)/\diff\lambda$ and $\diff T(\mathscr B_j)/\diff\lambda$ are the Radon--Nikodym derivatives of $T(\mathscr B_i)$ and $T(\mathscr B_j)$, respectively.
\end{definition}

Intuitively, $\mbox{KLD}(\cdot,\cdot)$ measures the relative entropy (or information gain) when comparing statistical models of inference; while, $\mbox{H}^2(\cdot,\cdot)$ is the probabilistic analog of the Euclidean distance between two distributions.  The latter is a natural measure for the Euclidean-space vector representation of barcodes.  The Kullback--Leibler divergence and Hellinger distance are nonnegative quantities and, in this context, take the value of zero if and only if two probability distributions are exactly equivalent.  Derivable from the Cauchy--Schwarz inequality, the Hellinger distance also has the desirable property of satisfying the unit interval bound for all probability distributions.  Namely, 
\begin{align*}
0\le\mbox{KLD}\big(T(\mathscr B_i)\,\|\,T(\mathscr B_j) \big)<\infty; \quad \quad 0\le \mbox{H}\big( T(\mathscr B_i), T(\mathscr B_j) \big)\le1.
\end{align*}
Note that it can be shown that when the probability measures $T(\mathscr B_i)$ and $T(\mathscr B_j)$ stem from the exponential family, the solutions of both integrals in \eq{eqn:kld} and \eq{eqn:hellinger} have a closed form \citep[e.g.,][]{van2000asymptotic,Duchi:2007aa}.  More specifically, for two normal random variables $T(\mathscr B_i)\sim \mathcal{N}(\mu_i,\sigma^2_i)$ and $T(\mathscr B_j)\sim \mathcal{N}(\mu_j,\sigma^2_j)$, we have the following:
\begin{align}
\label{eqn:kld_normal}
\mbox{KLD}\big(T(\mathscr B_i)\,\|\,T(\mathscr B_j) \big) &= \log\frac{\sigma_j}{\sigma_i}+\left[\frac{\sigma^2_i+(\mu_i-\mu_j)^2}{2\sigma^2_j}-\frac{1}{2}\right]\\
\label{eqn:hellinger_normal}
\mbox{H}^2\big( T(\mathscr B_i), T(\mathscr B_j) \big) &= 1-\sqrt{\frac{2\sigma_i\sigma_j}{\sigma^2_i+\sigma^2_j}}\exp\left\{-\frac{(\mu_i-\mu_j)^2}{4(\sigma^2_i+\sigma^2_j)}\right\}.
\end{align}
Based on the marginal normality assumption in \eq{eqn:marginal_normal} with parameter values $\mu_i, \mu_j$ and $\sigma_i, \sigma_j$ estimated empirically from the data, we show that the tropical coordinatization of persistence barcodes sufficiently preserves the fact that intra- and intersubtype samples do not cluster together.  This was done by calculating \eq{eqn:kld_normal} and \eq{eqn:hellinger_normal} for every pair $T(\mathscr B_i)$ and $T(\mathscr B_j)$ and encoding this information as similarity matrices $\mathbf{K}$ and $\mathbf{H}$, respectively.  Note that the key reasoning behind considering two different types of divergence measures is two illustrate the robustness of the results.  Figure \ref{Fig5} plots a transformed version of the matrix $\mathbf{K}^*=\mathbf{1}\mathbf{1}^{\T}-\exp\{-\mathbf{K}\}$ (for the purpose of exact comparisons and interpretability), where $\mathbf{1}$ is a vector of ones, and Hellinger distance matrix $\mathbf{H}$.  These results can be interpreted as follows: a value of $0$ represents exact likeness, while $1$ represents complete dissimilarity.  As expected from the literature, the intersubtypes are seen to be more homogeneous and demonstrate a higher level of similarity, while the intrasubtypes are comparatively much more random \citep[e.g.,][]{Yan:2011aa,Marshall:2013aa,pmid28034300}.  Moreover, there is very little overlap between the two groups.  These results are consistent for both $f$-divergence measures.

\section{Discussion}
\label{sec:discussion}

In this paper, we explored a functional vectorization method for persistence barcodes that is a topological embedding into Euclidean space based on tropical geometry.  We proved that it generates sufficient statistics for the family of all probability measures on the space of barcodes under mild regularity conditions.  The statistical and data analytic utility of this result lies in the fact that sufficiency allows parametric probabilistic assumptions to be imposed, which previously has been difficult due to the prohibitive geometry of the space of barcodes.  This allows for the application of classical parametric inference methodologies to persistence barcodes, which we demonstrated in a concrete example in dimension 1 persistence applied to avian influenza data.  Our methodological approach is based on the assumption that the random barcodes we obtain from real data may be represented by a parametric model.  Given the complex representation of persistent homology, we recognize that such an assumption may not always be reasonable---in which case, nonparametric methods as a more flexible alternative option \citep[e.g.,][]{1704.08248}.  However, note that the trade-off in relinquishing any parametric assumption is the extent to which inference may be carried out, which is a well-known problem in statistical modeling \citep[e.g.,][]{Crawford:2017aa}.

Our sufficiency also result provides the foundations for future research towards a better understanding of parametric probabilistic behavior on the space of barcodes. Since the tropical vectorization method was shown to be an injective function, we may take a class of parametric probability distributions (e.g., the exponential family) and calculate its image via the inverse of the tropical functions in order to explore what functional form the exponential family assumes on the space of barcodes. Such a study would be algorithmic in nature, since it would entail exploring collections of maps given by subsets of the orbits under the symmetric group action.

In terms of an increased utility in statistical analyses, the tropical functions are limited since the vector representation is unique up to barcodes only and not on the level of bars. This is restrictive in questions motivated by the application of our paper. For example, when predicting or modeling genetic outbreaks of infectious diseases, individual sequences (or at least the persistence generators) would need to be identifiable. Future research towards these efforts may begin with the development of a vectorization method that traces back to the individual bars of a barcode.


\section*{Software and Data Availability}

Software to compute and evaluate the tropical functions on barcodes is publicly available in C++ code, coauthored by Melissa McGuirl and Steve Oudot and located on the Tropix GitHub repository at \url{https://github.com/lorinanthony/Tropix}. The persistence barcodes were calculated using Ripser \citep{ripser}, which is written in C++ and is freely available at \url{https://github.com/Ripser/ripser}.  The bottleneck distances were computed using Hera \citep{doi:10.1137/1.9781611974317.9}, which is also written in C++ and freely available at \url{https://bitbucket.org/grey_narn/hera}.  The data used in this paper were obtained from publicly available sources and preprocessed, as detailed in Appendix \ref{appendix:data}. The final version used in the analyses in this paper are also publicly available on the Tropix GitHub repository.


\section*{Acknowledgements}

The authors are very grateful to Karen Gomez-Inguanzo and Melissa McGuirl for help with formulation of the code.  We also wish to thank Robert Adler, Ulrich Bauer, Omer Bobrowski, Iurie Boreico, Joseph Minhow Chan, Pawe\l \ D\l otko, Peter Hintz, Hossein Khiabanian, Albert Lee, Sayan Mukherjee, Ra\'ul Rabad\'an, and Nicole Solomon for helpful discussions.  We are especially indebted to Steve Oudot for his extensive input and support in code formulation and content throughout the course of this project.

A.M.~and J.A.P.-G.~are supported by the National Institutes of General Medical Sciences (NIGMS) of the National Institutes of Health (NIH) under award R01GM117591. L.C.~would like to acknowledge the support of start up funds from Brown University. Any opinions, findings, and conclusions or recommendations expressed in this material are those of the author(s) and do not necessarily reflect the views of any of the funders. The authors would also like to acknowledge GenBank, the HIV Sequence Database operated by Los Alamos National Security, LLC, and the National Center for Biotechnology Information (NCBI) Influenza Virus Database initiatives for making the data in this study publicly available.


\appendix
\renewcommand{\thesection}{Appendix}
\renewcommand{\thesubsection}{A\arabic{subsection}}
\renewcommand{\theAppDefinition}{A\arabic{AppDefinition}}
\renewcommand{\theAppClaim}{A\arabic{AppClaim}}
\section{}

\subsection{Mathematical supplement}
\label{appendix:math}

\begin{AppClaim}
The polynomials that are a ring of algebraic functions on the space of barcodes, identified by \cite{algfn}, are not Lipschitz with respect to the Wasserstein $p$- and bottleneck distances.
\end{AppClaim}

\begin{proof}
Consider the following counterexample.  Take one barcode to be the diagonal $\Delta$ (i.e., intervals of length $0$) denoted by $\mathscr B_0$ and consider the family of barcodes with a single interval $\mathscr B_x = \{ [x, x+1] \}$.  The Wasserstein $p$- and bottleneck distances between these $\mathscr B_0$ and $\mathscr B_x$ is the same for all $x$, namely, $1/2$.  Consider a polynomial given by \cite{algfn}, e.g.,
$$
p_{2, 1}(x_1, y_1, x_2, y_2, \ldots) = \sum_i (x_i+y_i)^2 (y_i-x_i)^1,
$$
then
\[
\begin{array}{ccc}
p_{2, 1} (\mathscr B_0) =0 &\, \textrm{and} \,& p_{2, 1} (\mathscr B_x) = (2x+1)^2.
\end{array}
\]
Now, the difference $p_{2, 1} (\mathscr B_x) - p_{2, 1}(\mathscr B_0)$ tends to infinity as $x$ tends to infinity.  The bottleneck (and Wasserstein $p$-) distance, meanwhile, is constant at value $1/2$ the entire time.  Therefore, by definition, this function is not Lipschitz.  This particular case is problematic and shows that other functions proposed in \cite{algfn} are also not Lipschitz with respect to the Wasserstein $p$- and bottleneck distances. 
\end{proof}

\subsection{Molecular phylogenetic analysis}
\label{appendix:trees}

A phylogenetic analysis consists of studying the evolutionary relationships of a set of individuals (i.e., viral sequences), with the aim of extracting and analyzing their diversification history.  This type of analysis is usually performed by using an alignment of sequences as input.  Given this input, a {\em phylogenetic tree} is then produced---that is, a graphical representation of the evolutionary relationships among the individuals.  Common statistical methods of phylogenetic inference are maximum likelihood and Bayesian phylogenetics, which aim to maximize a statistical parameter (either the likelihood of the alignment or its posterior probability) by modeling the molecular evolutionary processes considered in the phylogenetic reconstruction using stochastic processes \citep[e.g.,][]{doi:10.1093/sysbio/syq010, doi:10.1093/sysbio/sys029}.  Other methods are distance-based, which may be used on the trees themselves as a means of reconstruction \citep[e.g.,][]{10.2307/1720651, doi:10.1093/oxfordjournals.molbev.a040454}, or on spaces of trees, which are commonly used to study and compare various trees \cite[e.g.,][]{BILLERA2001733}.  See \cite{pachter2005algebraic} and \cite{lin2018tropical} for further detail and a discussion on challenges of computational phylogenetic analysis.

\subsection{Data sourcing and preprocessing}
\label{appendix:data}

The respective virus data sets in this paper were obtained from public sources.  We give the details on their sources and preliminary data processing procedures below.

\paragraph{HIV.}

HIV polymerase sequences derived from patients included in the MSM HIV transmission cluster \citep{juan2} were retrieved from supplementary material made public on GenBank. Sequences derived from unrelated patients were obtained from the Los Alamos HIV database in October 2016.  Only sequences from the same subtype (HIV subtype B) spanning the polymerase region were considered. In order to ensure that these sequences were not epidemiologically related, redundant sequences were removed after conducting an initial clustering analysis with a specified genetic distance threshold of 5\%, using CD-HIT \citep{pmid20053844}.  All sequences were aligned using MAFFTv7 \citep{pmid23329690}.

\paragraph{Avian influenza.}

HA and NA genes of avian influenza A were downloaded from the Influenza Virus Database of the National Center for Biotechnology Information (NCBI).  The resulting gene data sets were aligned with MAFFTv7 \citep{pmid23329690}.  Concatenated sequences of both genes (derived from the same sample) were generated with the package {\tt ape} written in R \citep{doi:10.1093/bioinformatics/btg412}.  The multiple sequence alignments were trimmed with TrimAl \citep{pmid19505945} in order to remove regions of sparse homology (i.e., biologically shared ancestry).\\

\noindent In both viral examples, pairwise distances were obtained using PAUP* \citep{Swofford01paup*:phylogenetic} and were calculated by using the GTR + GAMMA (4 CAT) model, which is commonly used for studying HIV and influenza data \citep{Tian06012015, pmid27783600}.  Briefly, the generalized time reversible (GTR) model is an evolutionary model that considers variable base frequencies, where each pair of nucleotide substitutions occur at different rates \citep{doi:10.1146/annurev.ge.29.120195.002153}.  Combined with a gamma distribution, it also accounts for rate variation among sites \citep{pmid7713447}.  The use of a substitution model when calculating genetic distances, as carried out according to these procedures, leads to estimates that are assumed to be more biologically accurate.


\newpage
\section*{Figures}

\begin{figure}[H]
\centering
\includegraphics[scale=0.8]{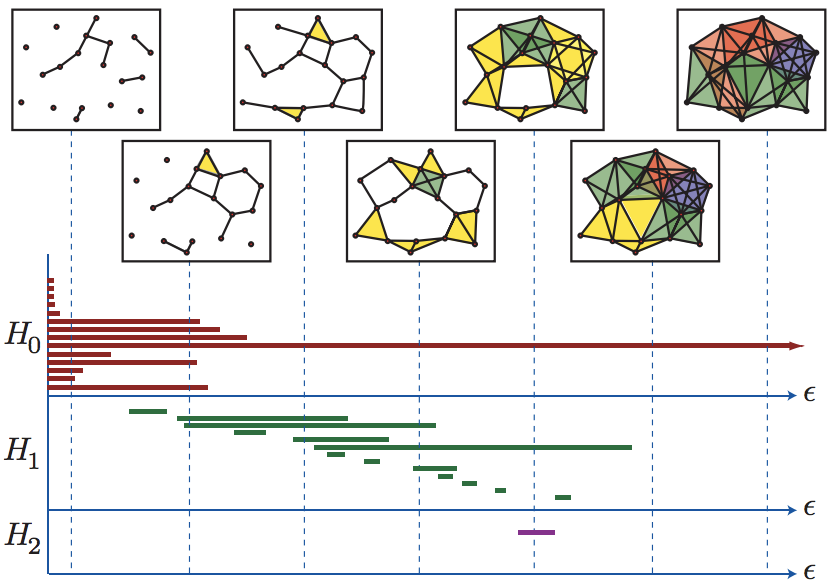}
\caption{Example illustrating persistent homology.  This figure originally appears in \cite{Ghrist2008}.  Here, 17 points are sampled from an annulus with topology given by $\beta_0 = 1$, $\beta_1 = 1$, and $\beta_2 = 0$.  Simplicial complexes are computed for continuous values of the filtration parameter, $\epsilon \in [0,\infty)$.  The filtration is illustrated in the upper panels as the evolution of a simplicial complex in terms of vertices, edges, and faces that are formed as the value of $\epsilon$ increases.  $H_0$ homology captures connected components; $H_1$ captures cycles whose boundaries are formed by edges between vertices; and $H_2$ captures cycles with boundaries formed by faces.  The dashed lines extending from the panels at instances of the filtration connect to the bars representing the topological features that exist at the corresponding values of $\epsilon$.  As $\epsilon$ progresses, connected components merge, cycles form and then fill up.  The barcode summarizes this progression of $\epsilon$ by tracking the ``lifetimes" of all topological features according to their homology groups.  Notice that there is a single $H_0$ bar that persists as $\epsilon \rightarrow \infty$, representing the single connected component of the annulus.  Also, there is a single $H_2$ bar representing the cycle bounded by faces in the corresponding panel, but the length of the bar is comparatively short and likely to be a spurious topological artefact.}
\label{Fig_A4}
\end{figure}

\begin{figure}[H]
\centering
\includegraphics[width = 0.8\textwidth]{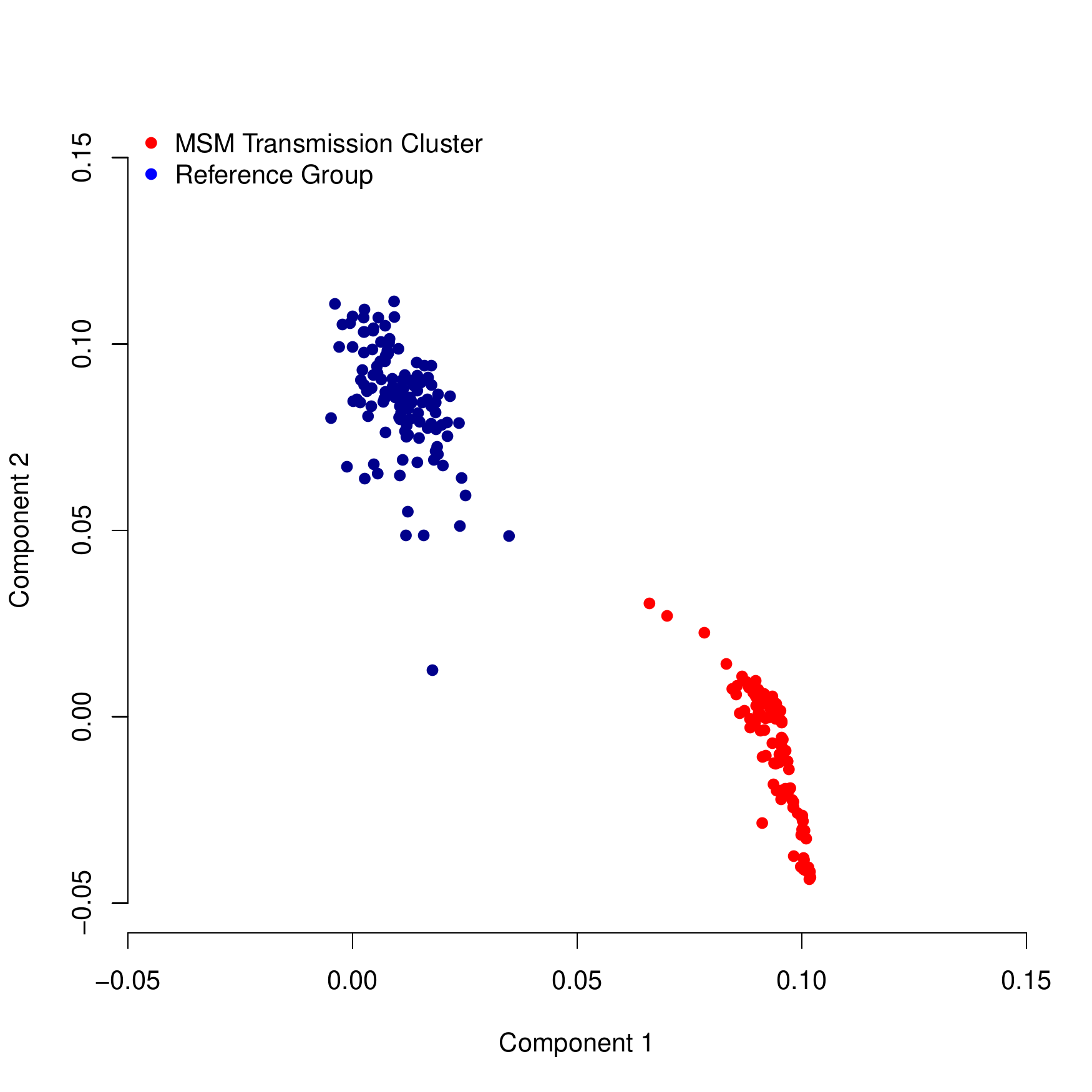}
\caption{{\bf Principal component analysis plot of genetic clustering for HIV sequences.}  Principal components calculated from pairwise genetic distances for HIV transmission cluster versus unrelated reference cluster.}
\label{Fig1}
\end{figure}

\begin{figure}[H]
\centering
\subfigure[]{
\includegraphics[width = 0.48\textwidth]{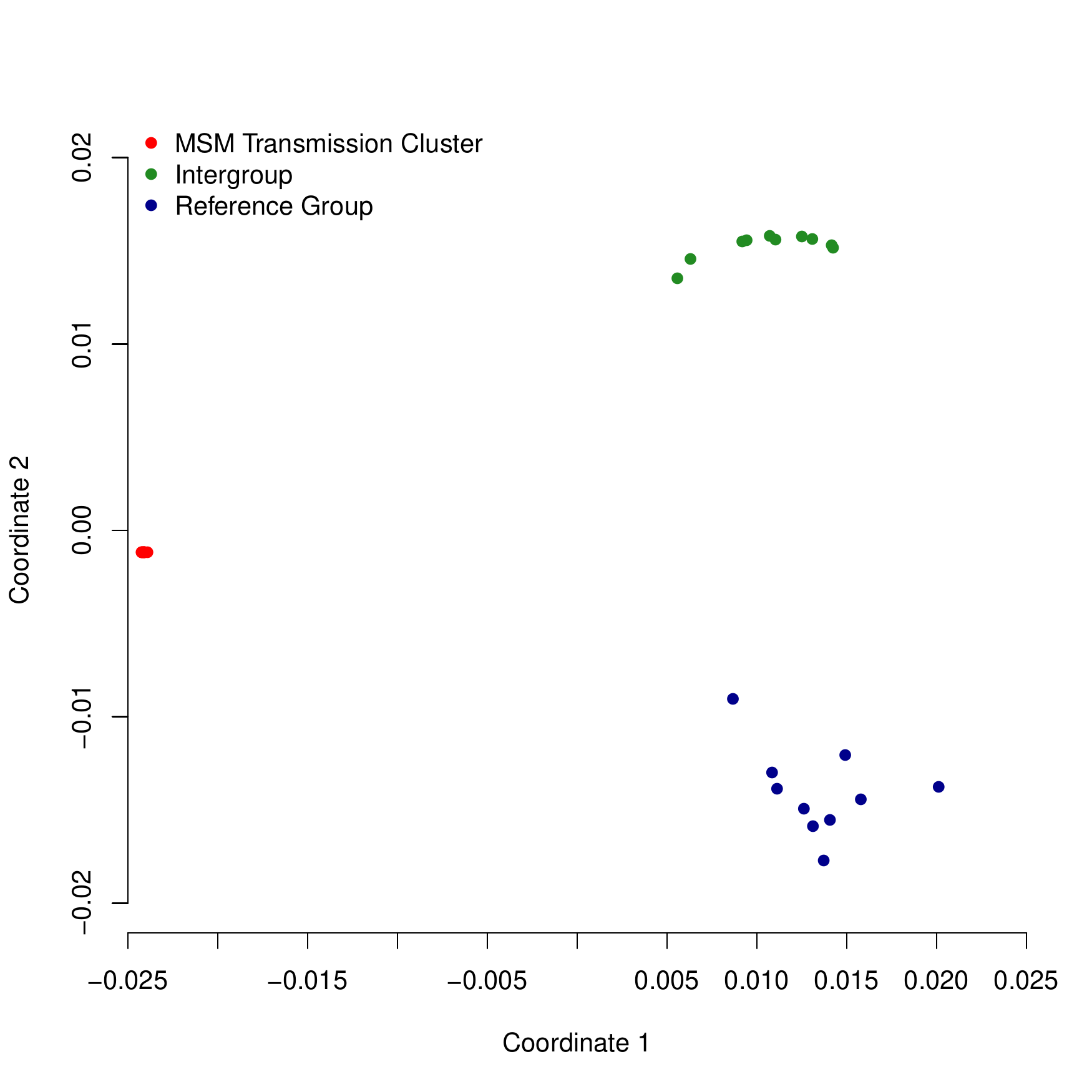}
\label{Fig2}
}
\subfigure[]{
\includegraphics[width = 0.48\textwidth]{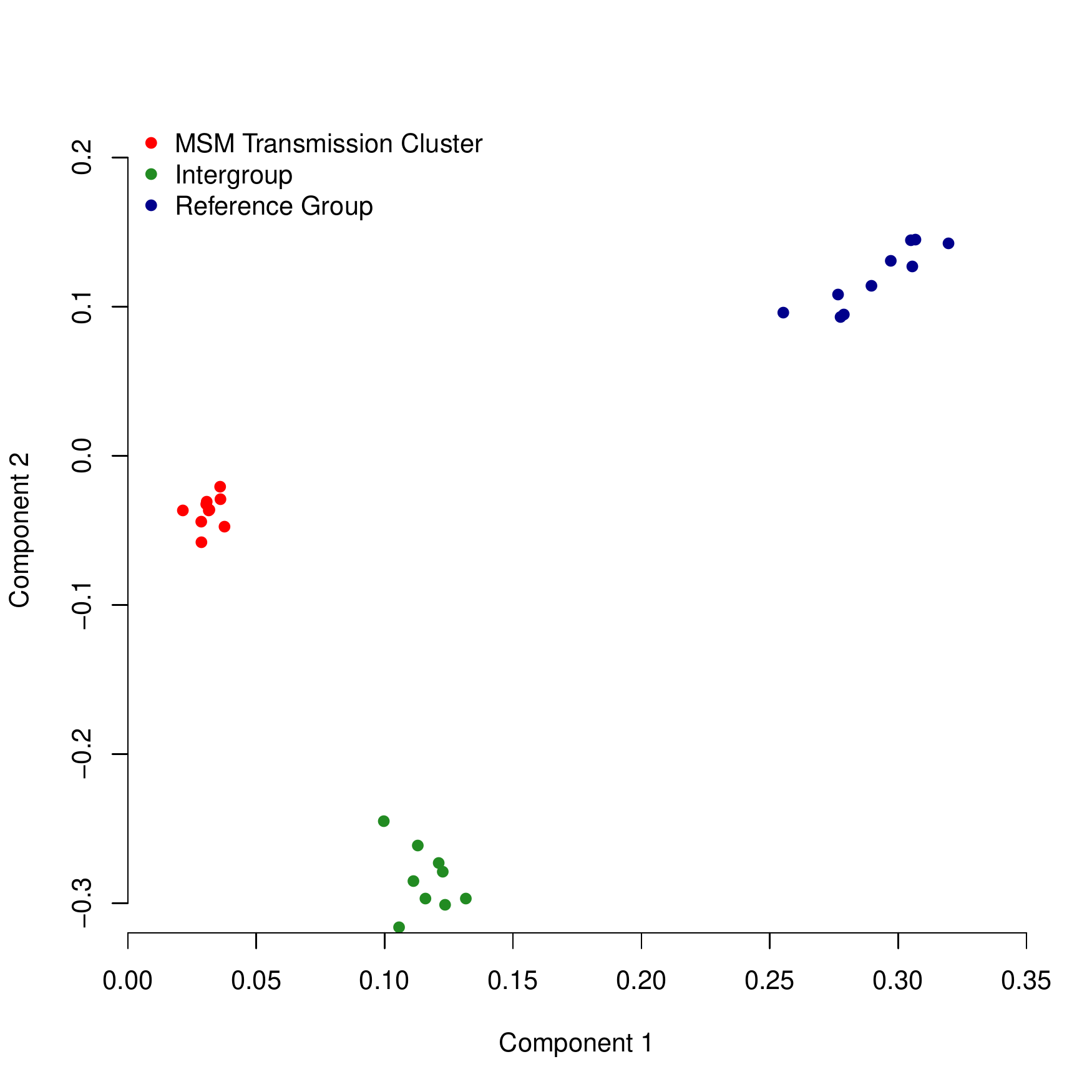}
\label{Fig3}
}
\caption{{\bf Projection plots for pairwise bottleneck distances and Euclidean distances for HIV sequences.}  (a) Metric multidimensional scaling (MDS) calculated from pairwise bottleneck distances calculated from dimension 0 persistence barcodes. (b)  Principal component analysis (PCA) calculated from pairwise Euclidean distances of tropical coordinatized barcodes.}
\label{newFig2}
\end{figure}

\begin{figure}[H]
\centering
\includegraphics[width = \textwidth]{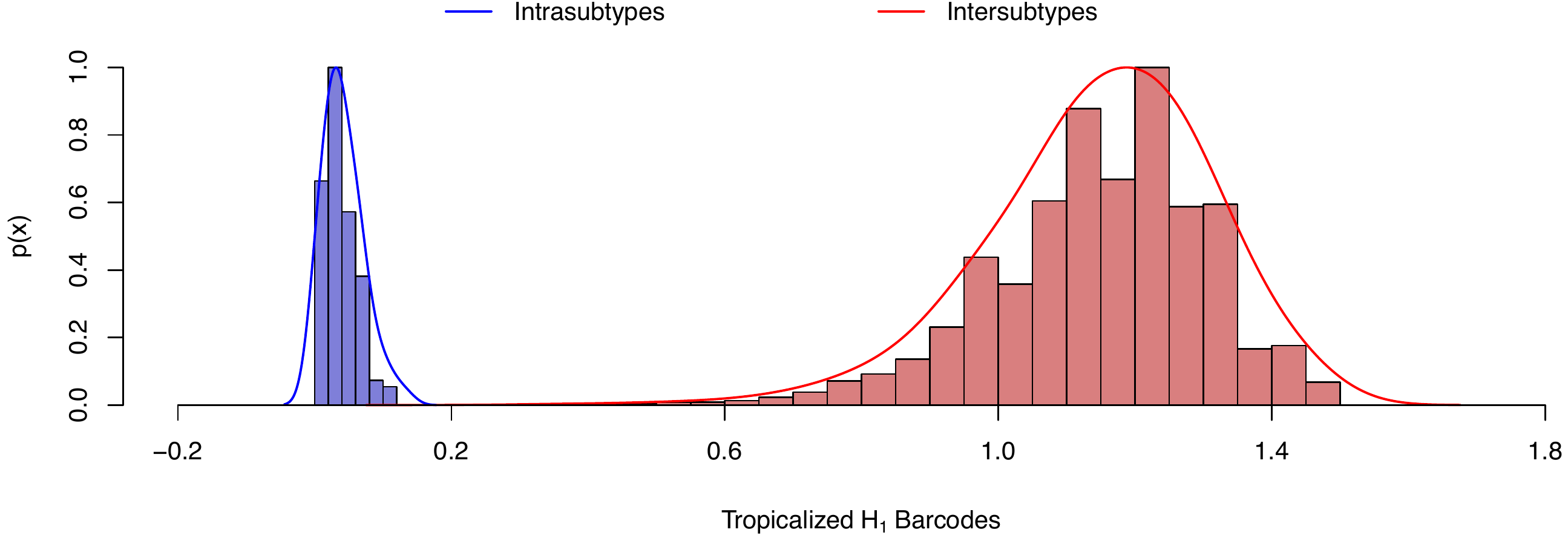}
\caption{{\bf Marginal distribution plot of intra- and intersubtype reassortment events for avian influenza.}  The marginal distributions for square-root transformations of the Euclidean barcode representations (via tropical coordinatization) were calculated for both intra- and intersubtype reassortment for avian influenza and then fitted with a smooth density function.}
\label{Fig4}
\end{figure}

\begin{figure}[H]
\centering
\subfigure[Scaled Kullback--Leibler Divergence]{
\includegraphics[width =0.7\textwidth]{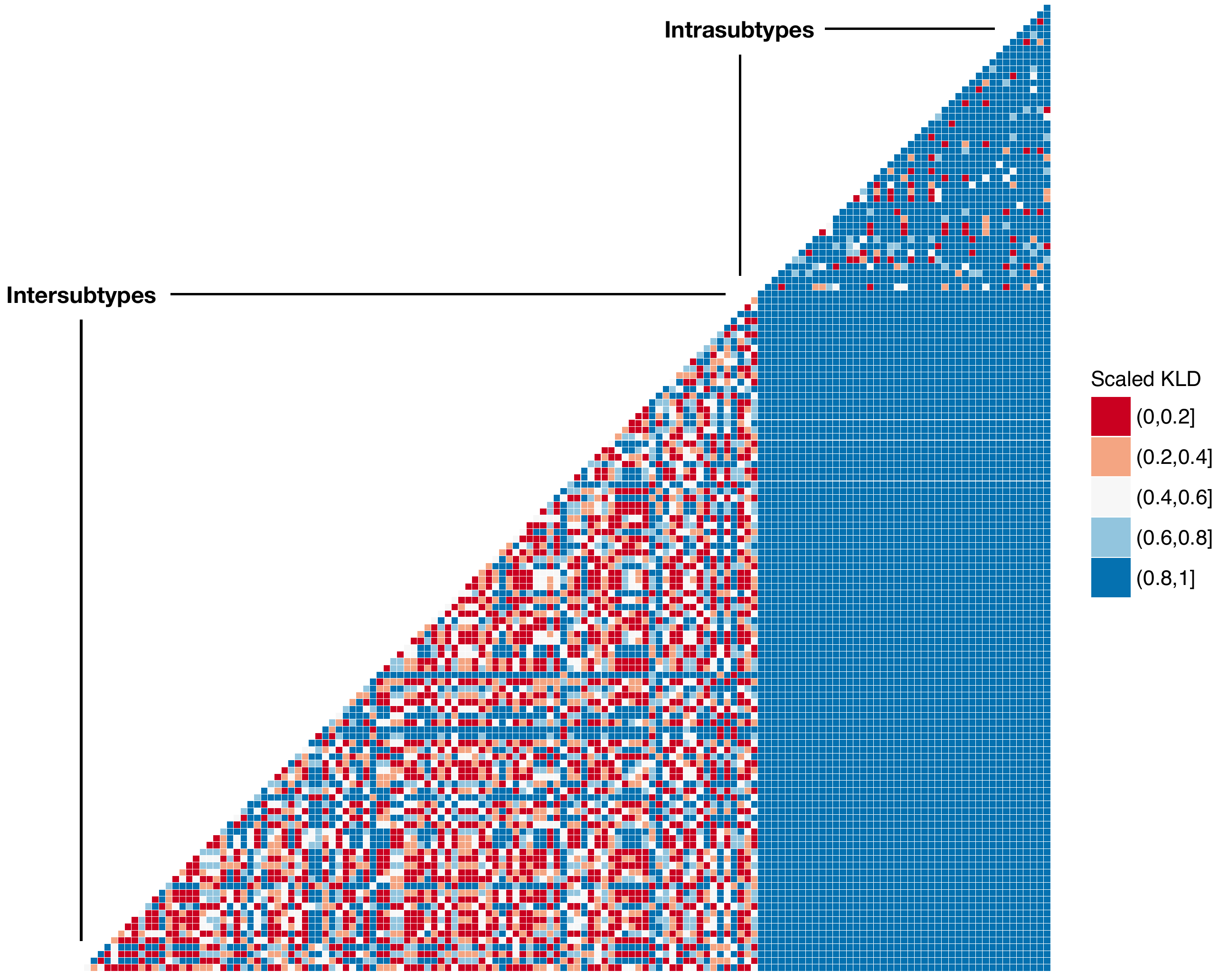}
}
\subfigure[Hellinger Distance]{
\includegraphics[width =0.7\textwidth]{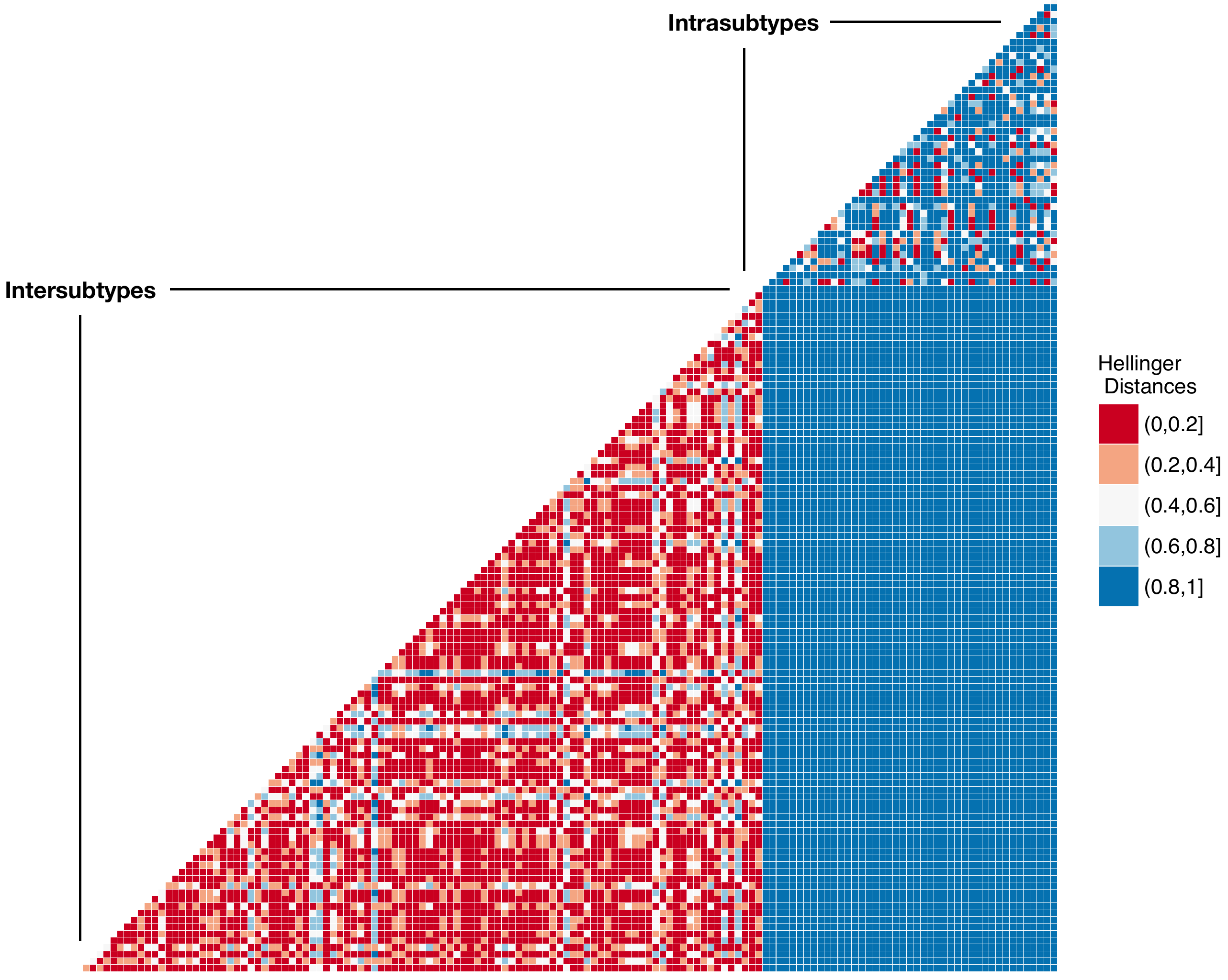}
}
\caption{{\bf Visual pairwise comparisons of transformed Kullback--Leibler divergence and Hellinger distances between Euclidean barcode representations.}  The entries are color-coded according to values of the corresponding similarity matrices (a) $\mathbf{K}^* = \mathbf{1}\mathbf{1}^{\T}-\mathbf{K}$, where $\mathbf{1}$ is a vector of ones, and (b) $\mathbf{H}$. A value of $0$ represents exact likeness, while $1$ represents complete dissimilarity. Red values represent those that are more similar, while blue values are more dissimilar. This plot shows that intersubtypes are more homogeneous (a higher level of similarity, i.e., more red), while the intrasubtypes are comparatively much more random. There is very little overlap between the two groups.}
\label{Fig5}
\end{figure}




\clearpage
\newpage
\bibliographystyle{chicago}  
\bibliography{Tropix_Ref} 


\end{document}